\documentclass[11pt]{article}
\usepackage{amsmath, amssymb, theorem, latexsym, epsfig, graphics, eufrak, floatrow}
\PassOptionsToPackage{linktocpage}{hyperref}
\usepackage[hyperindex,breaklinks]{hyperref}
\usepackage{subfigure}
\usepackage[a4paper, total={5in, 8in}]{geometry}
 \usepackage{ifpdf}
\numberwithin{equation}{section}
 
\theoremstyle{plain}
\theorembodyfont{\itshape}
\newtheorem{theorem}{Theorem}[section]
\newtheorem{proposition}[theorem]{Proposition}
\newtheorem{lemma}[theorem]{Lemma}
\newtheorem{corollary}[theorem]{Corollary}
                                                                                
\theorembodyfont{\rmfamily}
\newtheorem{definition}[theorem]{Definition}

\newtheorem{example}[theorem]{Example}
\newtheorem{remark}[theorem]{Remark}
\newtheorem{convention}[theorem]{Convention}
\newtheorem{problem}[theorem]{Problem}
\newtheorem{conjecture}[theorem]{Conjecture}
\newenvironment{proof}{{\noindent \textbf{Proof}\,\,}}{\hspace*{\fill}$\Box$\medskip}
\def\la{\lambda}

\def\mcv{\mathcal V}

\def\cc{\mathbb C}
\def\oc{\overline\cc}

\def\jos{\mathbf{Jos}}
\def\rr{\mathbb R}

\def\zz{\mathbb Z}
\def\La{\Lambda}
\def\cp{\mathbb{CP}}
\def\diag{\operatorname{diag}}

\def\re{\operatorname{Re}}

\def\mcg{\mathcal G}

\def\nn{\mathbb N}

\def\mcm{\mathcal M}

\def\mcc{\mathcal C}

\def\mcp{\mathcal P}
\def\wh#1{\widehat#1}

\def\var{\varepsilon}

\def\wt#1{\widetilde#1}

\def\tt{\mathbb T}

\title{On germs of constriction curves in model of overdamped Josephson junction, dynamical  isomonodromic foliation and Painlev\'e 3 equation}
\author{Alexey Glutsyuk\thanks{CNRS, UMR 5669 (UMPA, ENS de Lyon), France. E-mail: 
aglutsyu@ens-lyon.fr} \thanks{HSE University, Moscow, Russia} \thanks{Kharkevich Institute for Information Transmission Problems (IITP, RAS), Moscow}} 
\begin{document}
\maketitle
\vspace{-0.3cm}
{\it To my dear teacher Yu.S.Ilyashenko on the occasion of his 80-th birthday}
\begin{abstract} B.Josephson (Nobel 
Prize, 1973) predicted tunnelling effect for  a system of two superconductors 
separated by a  narrow dielectric (which is called {\it Josephson junction}): existence of a supercurrent through it and equations governing it. 
The {\it overdamped Josephson junction} 
is modeled by  a family of differential equations on 2-torus depending on 3
 parameters:  $B$ (abscissa), $A$ (ordinate), 
$\omega$ (frequency). We study its 
{\it rotation number} $\rho(B,A;\omega)$ 
as a function of  parameters.  
The {\it three-dimensional  phase-lock areas} are the  level sets $L_r:=\{\rho=r\}\subset\rr^3$ with non-empty 
interiors; they exist  for $r\in\zz$ (Buchstaber, Karpov, Tertychnyi). For every fixed $\omega>0$ 
and $r\in\zz$  the planar
slice $L_r\cap(\rr^2_{B,A}\times\{\omega\})$  is a garland of domains going vertically to infinity and separated by points; those separating points for which $A\neq0$ are called {\it constrictions.}  In a joint paper by Yu.Bibilo and the author, it was shown that 1) at each constriction the rescaled abscissa  $\ell:=\frac B\omega$ is integer and $\ell=\rho$; 2) 
the family $Constr_\ell$ of constrictions with  given $\ell\in\zz$ is an analytic submanifold in 
$(\rr^2_+)_{a,s}$, $a=\omega^{-1}$, $s=\frac A\omega$.  
In the present paper we show that 1) the limit points of $Constr_\ell$ are $\beta_{\ell,k}=(0,s_{\ell,k})$, where  
$s_{\ell,k}$ are the positive zeros of the $\ell$-th Bessel function $J_\ell(s)$; 2) to each 
$\beta_{\ell,k}$ accumulates exactly one  its component $\mcc_{\ell,k}$ (constriction curve), and it lands at   $\beta_{\ell,k}$ regularly. 
 Known numerical phase-lock area pictures show that high components of interior of each phase-lock area $L_r$ look  similar.   In his   paper with Bibilo,  
 the author introduced a candidate to the self-similarity map between neighbor components: the Poincar\'e map of the  dynamical isomonodromic foliation  
governed by Painlev\'e 3 equation. Whenever  well-defined, it preserves the  
 rotation number function. We show that the Poincar\'e map is well-defined on a neighborhood of the plane $\{ a=0\}\subset\rr^2_{\ell,a}\times(\rr_+)_s$, and it sends each constriction curve 
  germ $(\mcc_{\ell,k},\beta_{\ell,k})$ to  $(\mcc_{\ell,k+1},\beta_{\ell,k+1})$. 
\end{abstract}
 \tableofcontents
\section{Introduction and main results}
\subsection{Model of Josephson junction. Introduction and brief description of main results}
The tunnelling effect predicted by B.Josephson in 1962 \cite{josephson} (Nobel 
Prize, 1973) deals with a {\it Josephson junction:} a system of two superconductors 
separated by a  narrow dielectric. It states existence of a supercurrent through it and yields  equations governing it (discovered by Josephson). It was confirmed experimentally by P.W.Anderson and J.M.Rowell in 1963 \cite{ar}. 

The model of the so-called {\it overdamped Josephson junction},  
see \cite{stewart, mcc,  lev,  schmidt}, \cite[p. 306]{bar}, \cite[pp. 337--340]{lich}, 
\cite[p.193]{lich-rus}, \cite[p. 88]{likh-ulr} is described by the family of nonlinear differential equations
 \begin{equation}\frac{d\phi}{dt}=-\sin \phi + B + A \cos\omega t, \ \omega>0, \ B\geq0.\label{jos}\end{equation}
 Here $\phi$ is the  difference of phases (arguments) of the complex-valued 
 wave functions describing the quantum mechanic 
 states of the two superconductors. Its derivative is 
 equal to the voltage up to known constant factor.  
   
Equations (\ref{jos}) also arise in several models in physics, mechanics and geometry, e.g.,  
in  planimeters, see  \cite{Foote, foott}.

 The variable 
and parameter changes 
\begin{equation}\tau:=\omega t, \ \theta:=\phi+\frac{\pi}2, \ \ell:=\frac B\omega, \ a=\frac1\omega, \ s:=\frac A{\omega},\label{elmu}\end{equation}
 transform (\ref{jos}) to a 
non-autonomous ordinary differential equation on the two-torus $\mathbb T^2=S^1\times S^1$ with coordinates 
$(\theta,\tau)\in\rr^2\slash2\pi\zz^2$: 
\begin{equation} \frac{d\theta}{d\tau}=a\cos\theta + \ell + s \cos \tau.\label{jostor}\end{equation}
The graphs of its solutions are the orbits of the vector field 
\begin{equation}\begin{cases}  \dot\theta=a\cos\theta + \ell + s \cos \tau\\
 \dot \tau=1\end{cases}\label{josvec}\end{equation}
on $\mathbb T^2$. The {\it rotation number} of its flow, see \cite[p. 104]{arn},  is a function $\rho(B,A)$ of parameters\footnote{There is a misprint, 
missing $2\pi$ in the denominator, in analogous formulas in previous papers of the 
 author  with co-authors: \cite[formula (2.2)]{4}, \cite[the formula after (1.16)]{bg2}.}:
$$\rho(B,A;\omega)=\lim_{k\to+\infty}\frac{\theta(2\pi k)}{2\pi k}.$$
Here $\theta(\tau)$ is a general $\rr$-valued solution of the first equation in (\ref{josvec}) 
whose parameter is the initial condition 
for $\tau=0$. Recall that the rotation number is independent on the choice of the initial condition, see \cite[p.104]{arn}. 
 The parameter $B$ is called {\it abscissa,}  $A$ is called the {\it ordinate,} $\omega$ is called {\it frequency.} 
 Recall the following well-known definition. 

\begin{definition} \label{defasl} (cf. \cite[definition 1.1]{4}) The {\it $r$-th planar phase-lock area} is the level set 
$$L_r(\omega)=\{(B,A)\in\rr^2 \ | \ \rho(B,A;\omega)=r\}\subset\rr^2_{B,A},$$ 
provided that it has a non-empty interior. 
\end{definition}
The planar phase-lock areas were studied by V.M.Buchstaber, O.V.Karpov, S.I.Tertychnyi et al, see \cite{bibgl}--\cite{bt1}, 
\cite{LSh2009, IRF, krs, RK}, \cite{4}--\cite{gn19}, \cite{tert, tert2} and references therein.  The  following  results are known and proved mathematically:

1) Planar phase-lock areas exist only for integer rotation number  values 
({\it the rotation number quantization effect} observed and proved by V.M.Buchstaber, O.V.Karpov and S.I.Tertychnyi in \cite{buch2}, later also proved 
in  \cite{LSh2009, IRF}).

2) The boundary of each $L_r(\omega)$
 consists of two analytic curves, which are the graphs of two
functions $B=G_{r,\alpha}(A)$, $\alpha=0,\pi$, (see \cite{buch1}; this fact was later explained by A.V.Klimenko via symmetry, see \cite{RK}).

3)  The latter functions have Bessel asymptotics
\begin{equation}\begin{cases} G_{r,0}(A)=r\omega-J_r(-\frac A\omega)+O(\frac{\ln |A|}A) \\ 
G_{r,\pi}(A)=r\omega+J_r(-\frac A\omega)+O(\frac{\ln |A|}A)\end{cases}, \text{ as } A\to\infty\label{bessas}\end{equation}
 (observed and proved on physics level in ~\cite{shap}, see also \cite[p. 338]{lich},
 \cite[section 11.1]{bar}, ~\cite{buch2006}; proved mathematically in ~\cite{RK}); $J_r$ is the $r$-th Bessel function.
 
 4) Each planar 
 phase-lock area is a garland  of infinitely many bounded domains going to infinity in the vertical direction. 
 In this chain each two subsequent domains are separated by one point. This was  
 proved in \cite{RK} using the 
 above statement 3). Those  separation 
 points that lie on the horizontal $B$-axis, namely $A=0$, were calculated explicitly,  and 
 we call them the {\it growth points}, see \cite[corollary 3]{buch1}. The other separation points, which  lie outside the horizontal $B$-axis, are called the  {\it constrictions}. 
 
 5)  For every $r\in\zz$ and $\omega>0$ the $r$-th planar phase-lock area $L_r(\omega)$ is symmetric to the $-r$-th one with respect to the vertical 
 $A$-axis.
  
6) Every planar phase-lock area is symmetric with respect to the horizontal $B$-axis. See Figures 1--3 below.

7) In each planar phase-lock area $L_r(\omega)$ all its  constrictions lie in the same 
vertical line $\Lambda_r:=\{ B=r\omega\}$, see \cite[theorem 1.4]{bibgl}.

8)  Each constriction  $X\in L_r(\omega)$ is {\it positive:} the intersection of the interior of the area 
$L_r(\omega)$ with the vertical line $\Lambda_r$  contains a punctured neighborhood of the 
point $X$ in $\Lambda_r$,  see \cite[theorem 1.7]{bibgl}. 

9) For every fixed $\ell\in\zz$ the set of constrictions $(B,A;\omega)$ 
with abscissa $B=\ell\omega$ (i.e., constrictions  in $L_\ell(\omega)$), with variable 
$(A,\omega)$,  
is a {\it one-dimensional analytic submanifold} in 
$(\rr^2_+)_{A,\omega}$. We will denote it $Constr_\ell$ and represent it in the coordinates 
$(a,s)$: 
$$a:=\frac A{\omega}, \ \ s:=\frac1\omega,$$ 
$$Constr_\ell:=\{(a,s)\in\rr^2_+ \ | \ \left(\frac\ell a, \frac sa; a^{-1}\right) \text{ is a constriction}\}\subset(\rr^2_+)_{a,s}.$$

See Figures 1--3 of planar phase-lock areas and their  constrictions for 
$\omega=1,  0.5, 0.3$. 
 
 \begin{figure}[ht]
  \begin{center}
   \epsfig{file=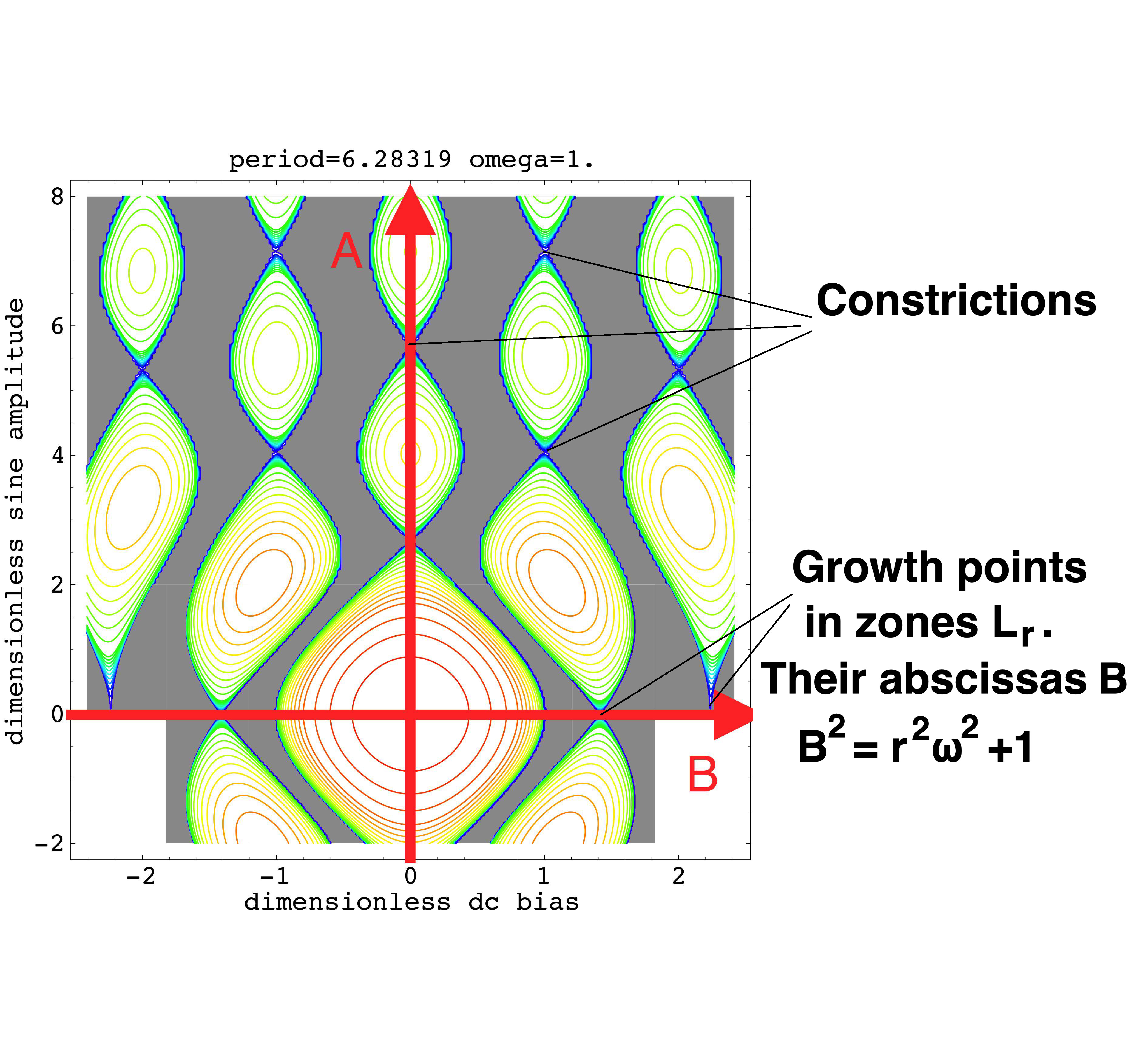, width=27em}
    \caption{Phase-lock areas and their constrictions for $\omega=1$. The abscissa is $B$, the ordinate is $A$. Figure 
    taken from paper \cite[fig. 1b)]{bg2} with authors' permission, with coordinate axes added.}
  \end{center}
\end{figure} 

 \begin{figure}[ht]
  \begin{center}
   \epsfig{file=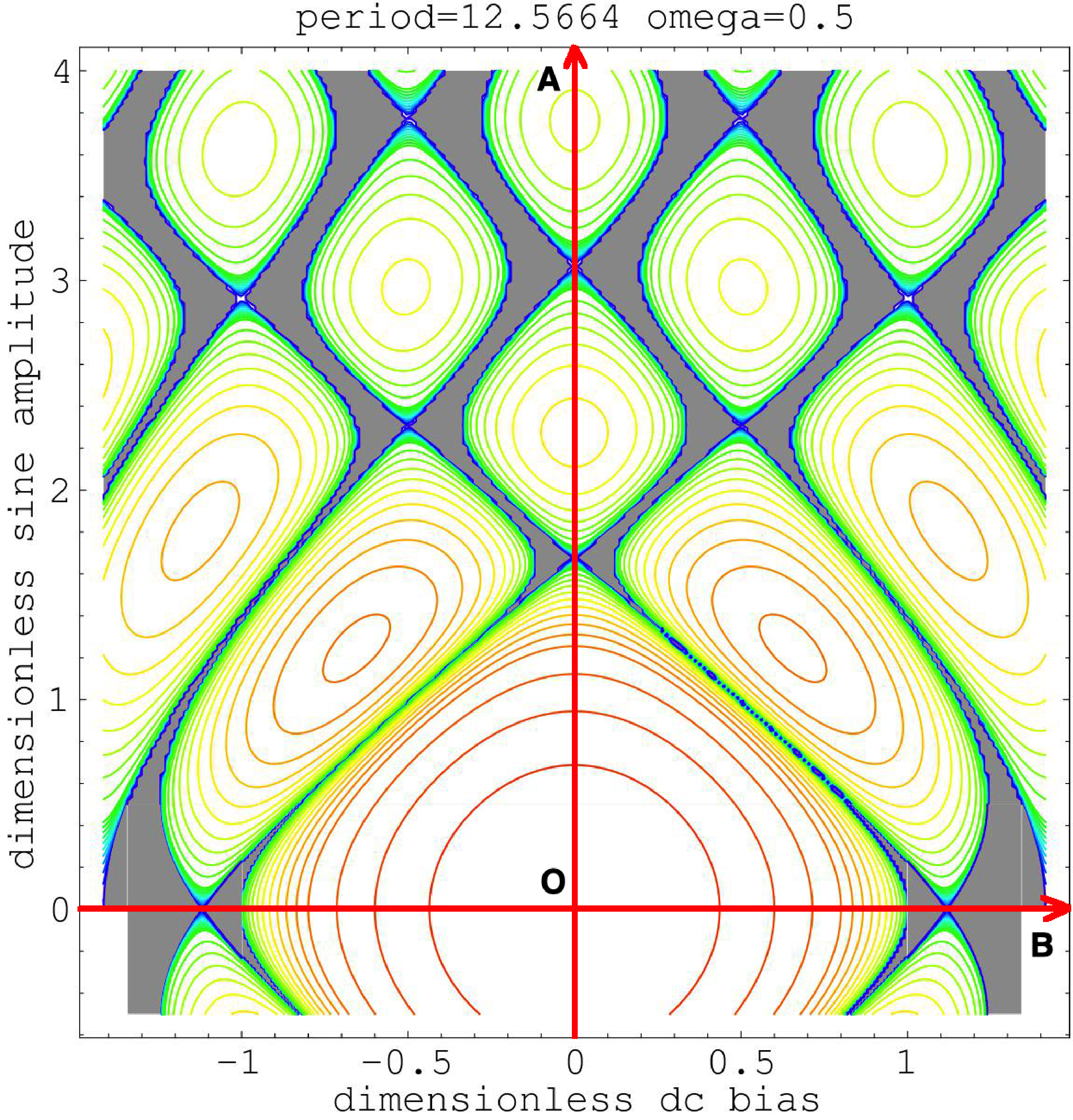, width=20em}
    \caption{Phase-lock areas and their constrictions  for $\omega=0.5$. Figure taken from papers \cite[fig. 1d)]{bg2}, 
    \cite[p. 331]{bt1} with authors' permission, with coordinate axes added.}
    \end{center}
\end{figure} 

\begin{figure}[ht]
  \begin{center}
   \epsfig{file=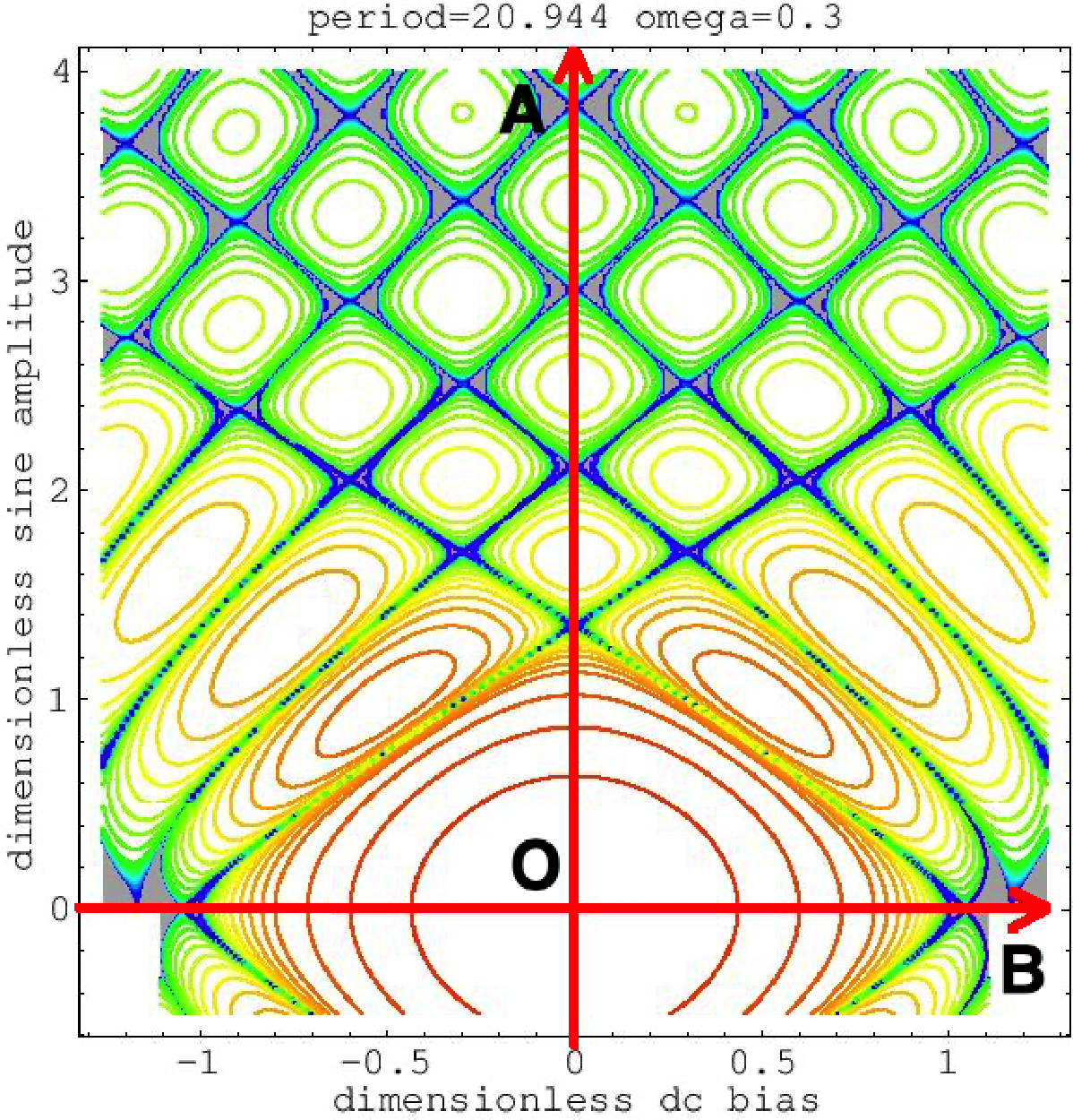, width=20em}
    \caption{Phase-lock areas and their constrictions  for $\omega=0.3$. Figure 
    taken from paper \cite[fig. 1e)]{bg2} with authors' permission, with coordinate axes added.}
     \end{center}
\end{figure} 
For a survey of other results on model (\ref{jos}) of overdamped Josephson junction and related topics see the paper \cite{bibgl} and its bibliography.
\begin{definition} A {\it constriction curve} is a connected component of the above constriction submanifold 
$Constr_\ell\subset(\rr^2_+)_{a,s}$. The  corresponding curve  
 in the parameters $(B, A; \omega)$, $B=\ell\omega=\frac\ell a$, $A=\frac sa$, $\omega=a^{-1})$ (which is a family of constrictions lying in $L_\ell(\omega)$) will be also called a {\it constriction curve}.
 \end{definition}
 
It would be interesting to study the {\it three-dimensional  phase-lock areas:}  
those level sets 
$$L_r=\{(B,A;\omega)\in\rr^2_{B,A}\times(\rr_+)_\omega \ | \ \rho(B,A;\omega)=r\}$$
 for which $Int(L_r)\neq\emptyset$. 
\begin{remark} The Buchstaber--Karpov--Tertychnyi  rotation number quantization effect also holds for the above three-dimensional areas $L_r$: they exist only for integer values of the rotation number $r$. This 
follows immediately from their analogous result and  its proof in two dimensions \cite{buch2}. 
 One has 
\begin{equation}L_r=\cup_{\omega>0}(L_r(\omega)\times\{\omega\}),\label{lrom}\end{equation}
where  $L_r(\omega)$ are the planar phase-lock areas for given $\omega$. 
The boundary $\partial L_r$ is the union of graphs  $\{ B=g_{r,\alpha}(A,\omega)\}$,  of 
functions $g_{r,\alpha}$, $\alpha\in\{0,\pi\}$, analytic on $\rr\times\rr_+$, whose restrictions to 
each line $\omega=const$ coincide with the corresponding functions $G_{r,\alpha}$. The 
proof of this fact repeats  the proof of the analogous statement for planar  
phase-lock areas, see \cite{buch1, RK}. Thus, 
\begin{equation} Int(L_r)=\cup_{\omega>0}Int(L_r(\omega))\times\{\omega\}.\label{intint}
\end{equation}
This implies that {\it the interior of each three-dimensional phase-lock  area $L_r$ is a union of domains separated by constriction curves} or the {\it growth point curve} $\{ A=0, \ B=\sqrt{r^2\omega^2+1}\}$.
\end{remark}

\begin{problem} Describe the asymptotic behavior of the  constriction curves.
\end{problem}

It follows from \cite[lemma 4.14 and arguments on pp. 5459--5460]{bibgl} that {\it the constriction 
submanifold $Constr_\ell$ accumulates to no point of the $a$-axis.} 

One of the main results of the present paper is Theorem \ref{cangerms} (Subsection 1.2).  It states that  the limit set of the constriction submanifold $Constr_\ell$ is the infinite sequence of   the points  
$\beta_{\ell,k}=(0,s_{\ell,k})$ of the $s$-axis,  
where $s_{\ell,1}, s_{\ell,2},\dots$ are the positive zeros of the Bessel function $J_\ell$;  
 exactly one constriction curve, denoted by $\mcc_{\ell,k}$, accumulates to each $\beta_{\ell,k}$ (which is its unique limit point); each $\mcc_{\ell,k}$  lands 
at $\beta_{\ell,k}$ regularly. Using Theorem \ref{cangerms}, we deduce the 
following theorem.

\begin{theorem} \label{thgarl} The interior of each three-dimensional phase-lock area has 
infinitely many components. 
\end{theorem}

\begin{conjecture} \label{conjeach} Each constriction curve 
lands at some $\beta_{\ell,k}$, i.e., coincides with some $\mcc_{\ell,k}$. 
\end{conjecture}

\begin{conjecture} \label{conj3d} \cite[conjecture 6.8]{bibgl} Each constriction curve is bijectively analytically projected onto the $a$-axis $(\rr_+)_a$ (or equivalently, onto $(\rr_+)_\omega$). 
The corresponding conjectural arrangement of constriction curves and components of the phase-lock areas $L_\ell$ is presented at Fig. \ref{figc}a).  
\end{conjecture}

\begin{definition} A constriction curve different from $\mcc_{\ell,k}$ (if any) is called a {\it queer constriction curve.} 
\end{definition}

We prove Proposition \ref{conjs} (Subsection 1.2), which states that a constriction curve is queer, if and only if 
the restriction to it of the coordinate $a$ is unbounded on both its sides. This will show that  Conjecture \ref{conj3d} would imply Conjecture \ref{conjeach}. 

We deduce Corollary \ref{cqueer} (of Proposition \ref{conjs}), which states that for every queer constriction curve $\mcc$ in $L_\ell$ (if any) there exists a connected component of 
the interior  $Int(L_\ell)$ that is adjacent to $\mcc$ and is adjancent to no curve $\mcc_{\ell,k}$. 
This shows that {\it Conjecture \ref{conjeach} is equivalent to the statement saying that each 
component of $Int(L_r)$ is adjacent to some $\mcc_{\ell,k}$. }

\begin{remark} Conjecture \ref{conj3d} implies that as $\omega$ varies, in the corresponding  planar phase-lock areas $L_r(\omega)$  constrictions can neither be born, nor disappear, as $\omega$ 
crosses a value $\omega_0$ for which the boundary curves of $L_r(\omega_0)$ are tangent to each other at a constriction. Conjectural picture of the three-dimensional phase-lock areas 
and conjecturally impossible picture, with a queer constriction curve and the corresponding 
component of $Int(L_r)$ (called a {\it queer component}), are presented at Fig. 4a) and 4b) respectively. 
\end{remark}
\begin{figure}[ht]
  \begin{center}
   \epsfig{file=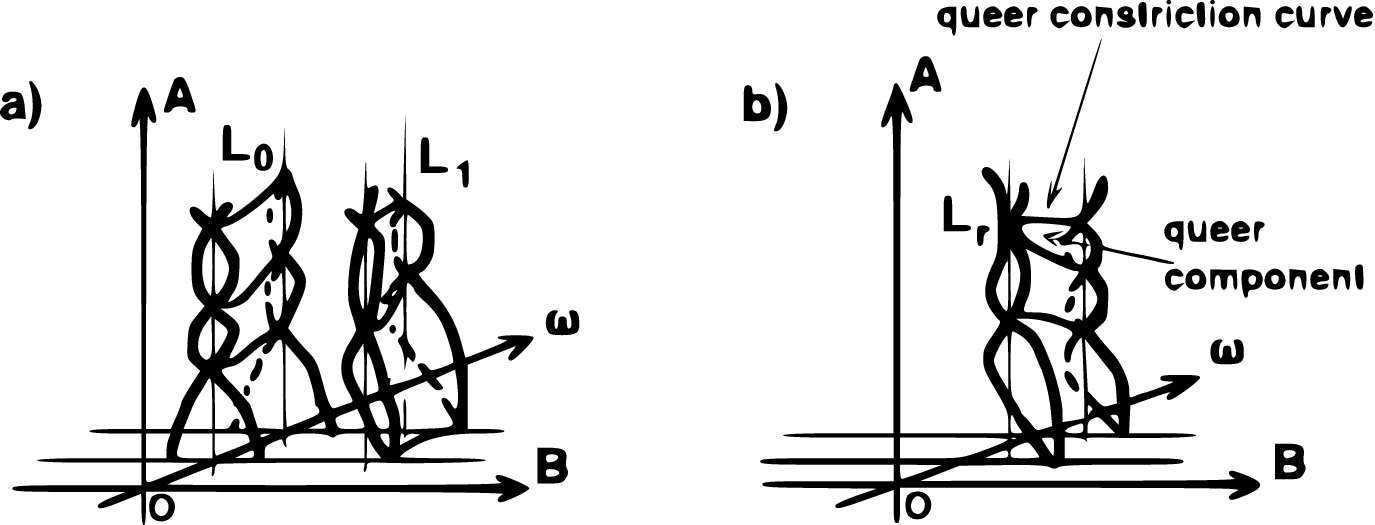, width=33em}
    \caption{Three-dimensional phase-lock areas and constriction curves: \newline a) the conjectural picture, without queer constriction curves; \newline b) a priori possible picture with a queer curve and a queer component.}\label{figc}
  \end{center}
\end{figure}

Each three-dimensional phase-lock area has a kind of {\it self-similarity structure:} any two its components lying high enough look similarly. It would be 
interesting to find and study a self-similarity map of the parameter space that sends one component to the next one, adjacent to it from above.  In his paper with Yu.Bibilo, 
see \cite[subsection 6.2]{bibgl}, the author suggested a candidate to such 
a self-similarity map: the {\it Poincar\'e map of the so-called dynamical isomonodromic foliation 
$\mcg$.} 
Namely, he introduced the following four-dimensional extension  of the three-dimensional family (\ref{josvec})  of dynamical systems 
on torus modeling overdamped Josephson junction:
\begin{equation}  \frac{d\theta}{d\tau}=\nu+a\cos\theta+s\cos\tau+\psi\cos(\theta-\tau); \ \ \ 
\nu,a,\psi\in\rr, \ \ s>0, \ (a,\psi)\neq(0,0).\label{gen3}\end{equation}
Family (\ref{jostor}) embeds to (\ref{gen3}) by the mapping 
\begin{equation}(\ell, s,\omega)\mapsto(\nu, a, s, \psi)=(\ell, \omega^{-1}, s, 0).\label{paremb}\end{equation}
 In the new coordinates $(\ell, \chi, a,  s)$ on the parameter space of family (\ref{gen3}), 
$$\chi=\frac{\psi}{2s}, \ \ell:=\nu-\frac{\psi a}{s}=\nu-2\chi a,$$ 
consider the line field given by the following system of non-autonomous differential equations introduced in \cite[formula (6.4)]{bibgl}:
  \begin{equation}\begin{cases}  \ell'_s=0\\ 
   \chi'_s=\frac{a-2\chi(\ell+2\chi a)}{2s}\\
  a'_s=-2s\chi+\frac as(\ell+2\chi a)\\
 \end{cases}.\label{isomnews}\end{equation} 
The  foliation of the four-dimensional parameter space by its phase curves, which are graphs of solutions of system (\ref{isomnews}), is called the {\it dynamical isomonodromic foliation} 
and denoted by $\mcg$. (This name comes from relation to isomonodromic families of linear systems; see explanation below, in Remark \ref{rkisom}). The author has shown in \cite[subsection 6.2]{bibgl} that each its leaf  is a family of flows on $\tt^2$ that are conjugated to each other by diffeomorphisms isotopic to the identity. Thus, {\it the rotation number and the number $\ell$ are constant along each leaf of the foliation $\mcg$.}   For every $\ell\in\rr$ the function 
  \begin{equation}w(s):=\frac{a(s)}{2s\chi(s)}=\frac{a(s)}{\psi(s)}\label{wsnew}\end{equation}
  satisfies Painlev\'e 3 equation 
  \begin{equation} w''=\frac{(w')^2}w-\frac{w'}{s}-2\ell\frac{w^2}{s}+(2\ell-2)\frac1s+w^3-\frac1w.\label{p3}\end{equation}
   along solutions of (\ref{isomnews}), see \cite[theorem 6.6]{bibgl}. 
   
   Consider family (\ref{jostor}) modeling Josephson junction as the subfamily in  
   (\ref{gen3}): the hyperplane  $\{\chi=0\}$, see (\ref{paremb}). A solution $(\ell, \chi(s), a(s))$ 
   with initial condition $(\ell, 0,a_0)$ in the latter hyperplane  at $s=s_0$ may return back to 
   the hyperplane at its another point $(\ell, 0, a_1)$ at some next moment $s_1>s_0$. If it happens, then 
   this induces the {\it Poincar\'e first return map} $\mcp:(\ell,a_0,s_0)\mapsto(\ell,a_1,s_1)$ 
   defined on some subset of the hyperplane $\{\chi=0\}$.  
   
   \begin{problem} (see \cite[problems 6.11--6.13]{bibgl}) Study the action of the 
Poincar\'e map $\mcp$ on the three-dimensional phase-lock area portrait of family (\ref{jostor}). Is it true that 
it sends each interior component of every phase-lock area (starting from the second component) 
 to its next component, adjacent to 
it from above? Is it true that it sends each constriction curve $\mcc_{\ell,k}$ to  
$\mcc_{\ell,k+1}$?
\end{problem}
\begin{problem}
Describe the subset  where the Poincar\'e map $\mcp$ is well-defined. Describe 
the subset where all its iterates are well-defined. 
\end{problem}

\begin{remark}
The second part of the latter problem is closely related to the description of those solutions of 
Painlev\'e 3 equation (\ref{p3}) that have an infinite lattice of residue one poles in $\rr_+$. 
See Subsection  1.5, where  relation to poles of solutions of Painlev\'e 3 equation 
will be presented together with a brief survey of results on poles.
\end{remark}

Our second main result  is Theorem \ref{tcgerm} (Subsection 1.3). It states that 

- the Poincar\'e 
map $\mcp$ is well-defined on a neighborhood of the plane $\{ a=0\}\subset\rr_{\ell}\times(\rr^2_+)_{a,s}$; 

- its restriction to the latter plane sends a point $(\ell,0,s)$ to $(\ell,0,s^*)$, $s^*>0$, where $s$, $s^*$ are neighbor zeros of a solution of the $\ell$-th Bessel equation. 

To study the domain of the  map $\mcp$, it is important to study analytic extensions 
of solutions of system (\ref{isomnews}) and of the Painlev\'e 3 equation (\ref{p3}). 
\begin{remark}
It is well-known that solutions $w(s)$ of Painlev\'e 3 equation (\ref{p3}) are  meromorphic on the universal cover of the punctured line $\cc^*_s=\cc\setminus\{0\}$ (classical result coming from the Painlev\'e property), and {\it all their poles are simple with residues $\pm1$} 
\cite[theorem 31.1]{GLS}. But a generic solution of (\ref{p3}) has a nontrivial monodromy and 
is not single-valued on $\cc^*$. 
\end{remark}

Our third main result is Theorem \ref{thext} (Subsection 1.4) classifying singularities 
$s_0\in\cc_s$ of 
solutions $(\chi(s), a(s))$ of system (\ref{isomnews}). It states that each solution is meromorphic on the universal cover over $\cc^*_s$, and each its pole $s_0$  is 
either a zero of the corresponding solution $w(s)$ of the Painlev\'e 3 equation (\ref{p3}) with unit 
derivative, or its pole with residue -1. Theorem \ref{thext} also states that for each initial condition 
lying in the constriction submanifold $Constr_\ell$ the corresponding solution $(\chi(s),a(s))$ is meromorphic either on $\cc^*$, or on its double cover, and $w(s)$ is meromorphic on $\cc^*$. 
In the case, when the initial condition lies on a constriction curve $\mcc_{\ell,k}$, Theorem \ref{thext} states that the  solution $(\chi(s),a(s))$ is meromorphic single-valued on $\cc^*_s$. 

Let us note the Hamiltonian nature of non-autonomous system (\ref{isomnews}). 

 \begin{proposition} For every fixed $\ell\in\rr$ the corresponding differential equation 
  (\ref{isomnews}) on vector function $(\chi(s),a(s))$ is Hamiltonian with 
  time $s$-depending Hamiltonian function 
  \begin{equation} H(\chi,a,s):=-\frac{\chi^2a^2}{s}+\frac{a^2}{4s}+s\chi^2-\frac{\ell\chi a}s.
  \label{hamilt}\end{equation}
  In particular, for every positive $s_1$, $s_2$, $s_1<s_2$, 
  the non-autonomous flow map of equation (\ref{isomnews}) from time $s_1$ to time $s_2$ 
  (on a domain in $\cc^2_{\chi,a}$ where it is well-defined) preserves the standard symplectic form $d\chi\wedge da$.
  \end{proposition}
  The proposition follows by straightforward calculation.
  \begin{remark} Representations of Painlev\'e  equations as a Hamiltonian systems were found by K.Okamoto \cite[p. 742]{ok}. For Painlev\'e 3, the Okamoto Hamiltonian is a quartic polynomial, 
  with the only quartic monomial being product of two squares, as in (\ref{hamilt}). But in difference to (\ref{hamilt}), it also contains  cubic and linear terms. 
  \end{remark}

     Family of dynamical systems (\ref{josvec}) modeling Josephson junction 
   can be equivalently described by a family of two-dimensional 
   linear systems of differential equations on 
   the Riemann sphere, see \cite{bkt1, buch2, bt1, Foote, LSh2009, IRF}, 
   \cite[subsection 3.2]{bg}. This is also true for the extended family (\ref{gen3}). 
   Namely, in the complex variables 
 $$\Phi=e^{i\theta}, \ \ z=e^{i\tau}$$ 
 equations (\ref{gen3}) can be equivalently written as Riccati equations 
\begin{equation}\frac{d\Phi}{dz}=\frac1{z^2}\left(\frac s2\Phi+\frac\psi2\Phi^2\right)+\frac1z\left(\nu\Phi+\frac a2(\Phi^2+1)\right) 
+\left(\frac s2\Phi+\frac\psi2\right).\label{ricg}\end{equation}
A function $\Phi(z)$ is a solution of the latter Riccati equation, if and only if 
$\Phi(z)=\frac{Y_2(z)}{Y_1(z)}$, where $Y=(Y_1,Y_2)(z)$ is a solution of the linear system 
\begin{equation} Y'=\left(-s\frac{\mathbf K}{z^2}+\frac{\mathbf R}z+s\mathbf N\right)Y,\label{mchoy}\end{equation}
$$\mathbf K=\left(\begin{matrix}\frac12 & \chi \\ 0 & 0\end{matrix}\right), \ 
  \mathbf R=\left(\begin{matrix}-(\ell+\chi a) & -\frac{a}2\\ \frac{a}2 & \chi a\end{matrix}\right),  \ 
  \mathbf N=\left(\begin{matrix}-\frac12 & 0 \\ \chi  & 0\end{matrix}\right); \ \ \ 
   \ \chi=\frac\psi{2s}.
  $$
  
 Given a $z_0\in\cc^*=\cc\setminus\{0\}$, say, $z_0=1$, the space of germs of solutions of 
 linear system (\ref{mchoy}) at $z_0$ is identified with the vector space $\cc^2$ of initial 
 conditions at $z_0$. Analytic extension along a counterclockwise circuit around the origin 
 is a linear operator acting on the latter local solution space, called the {\it monodromy operator}.  
   
 \begin{remark} \label{rkisom}
  Solutions of (\ref{gen3}) correspond to isomonodromic families of linear systems 
 (\ref{mchoy})  \cite[theorem 6.6]{bibgl}. (This explains the name "isomonodromic foliation" for the foliation by graphs of solutions of (\ref{isomnews}).)   The latter isomonodromic families are induced from classical Jimbo isomonodromic deformations given in \cite{J}, governed by Painlev\'e 3 equation \cite[pp. 1156--1157]{J}, which implies equation (\ref{p3}). It is known that {\it linear 
   systems (\ref{mchoy}) corresponding to constrictions in subfamily (\ref{jostor}) of 
   (\ref{gen3}) have trivial monodromy,} see \cite[proposition 4.1]{bibgl}. 
   For a necessary background on linear systems with irregular singularities  and isomonodromic families see, for example, \cite[sections 2 and 3.1]{bibgl} and Subsection 2.7 below.    \end{remark}
   
   The plan of proofs of main results is presented in Subsection 1.6.

\subsection{The constriction manifold: limit points and landing components}

 \begin{theorem} \label{cangerms} 1) The limit points of the constriction submanifold 
 $Constr_\ell\subset(\rr_+^2)_{a,s}$ are the points 
 $$\beta_{\ell,k}=(0,s_{\ell,k}), \ \ s_{\ell,k}>0 \ \text{ is the } \ k-\text{th positive zero of the Bessel function}$$
 \begin{equation}  J_\ell(s)=\frac1\pi\int_0^\pi\cos(\ell\tau-s\sin\tau)d\tau.\label{besself}\end{equation}

2) For every $\beta_{\ell,k}$ there exists a unique connected component $\mcc_{\ell,k}$ of the manifold $Constr_\ell$   
landing  at $\beta_{\ell,k}$, and $\mcc_{\ell,k}$ accumulates  to no other $\beta_{\ell,j}$.

 3) The union of the closure $\overline{Constr_\ell}$ in 
$\rr_{\geq0}\times\rr_s$ and its  image under the symmetry with respect to the $s$-axis is  
a regular connected analytic 
submanifold  $\widehat{Constr}_\ell\subset\rr^2_{a,s}$. It intersects the $s$-axis  
orthogonally at points  $\beta_{\ell,k}$. 

4) The image of each curve $\mcc_{\ell,k}$ under the projection to the $a$-coordinate is the whole $a$-semiaxis $\rr_+$.
\end{theorem}

\begin{proposition} \label{conjs} Let $\mcc\subset Constr_\ell$ be a constriction curve. 
Let the restriction to $\mcc$ of the coordinate function $a$ be  bounded on at least one its side: we split $\mcc$ into two semi-curves separated by a point, and we require that $a$ be bounded on at least one of them.  Then $\mcc$ coincides with some of 
$\mcc_{\ell,k}$.
\end{proposition}

\begin{corollary} Conjecture \ref{conj3d} implies Conjecture \ref{conjeach}.
\end{corollary}

For a given $\ell\in\rr$ and a subset $U\subset(\rr_+^2)_{a,s}$ set 
\begin{equation}\wt U_{\ell}:=\{(B,A;\omega)=(\frac{\ell}a, \frac sa, a^{-1}) \ | (a,s)\in U\}\subset
\rr_B\times(\rr_+^2)_{A,\omega}.\label{uphase}\end{equation}

\begin{corollary} \label{cqueer} 
 Let $\ell\in\zz$.  Every  queer constriction curve $\mcc\subset Constr_\ell$ (a component 
 different from $\mcc_{\ell,k}$'s, if any) satisfies the following statements:

(i) The coordinate function 
$a$ is unbounded from above on both sides of $\mcc$ (i.e. on both  semicurves in $\mcc$ separated by one point);

(ii) Let $U_{\mcc}\subset(\rr_+^2)_{a,s}$ denote the connected component of the complement 
$\rr_+^2\setminus\mcc$ lying on the right from $\mcc$. There exists a  domain 
$V\subset U_{\mcc}$ adjacent to $\mcc$ such that the two-dimensional domain 
$\wt V_{\ell}\subset\{ B=\ell\omega\}$, see (\ref{uphase}), lies in  a connected component of the interior of the 
three-dimensional phase-lock area $L_\ell$. The latter 
component of $Int(L_\ell)$ is called a {\bf queer component} adjacent to $\mcc$ from the right. 
It is  adjacent to no curve $\mcc_{\ell,k}$ (but a priori, may be adjacent to some other queer constriction curves). See Fig. 4b). 
\end{corollary}

\subsection{The  Poincar\'e map of the isomonodromic foliation $\mcg$}
Recall that the foliation $\mcg$ is the foliation by graphs of solutions of system (\ref{isomnews}). 
  The initial 3-parametric family (\ref{josvec}) of dynamical systems on $\tt^2$ modeling Josephson junction is a 
cross-section to $\mcg$. See (\ref{paremb}). In the parameters $(\ell, \chi, a, s)$, this is the 
hyperplane $\{\chi=0\}$ with the plane $\{ a=0\}$ deleted.  We denote it by $\jos^o$ and its closure 
by $\jos$:  
$$\jos:=\{\chi=0, s>0\}\subset\rr^4_{\ell, \chi, a, s}, \ \ \jos^o:=\jos\setminus\{ a=0\}.$$
The foliation $\mcg$ is tangent to $\jos$ at points of its proper hyperplane $\{ a=0\}$. 
   
  \begin{definition} Consider the solution $(\ell,\chi(s),a(s))$ of 
  equation (\ref{isomnews}) whose graph  passes through a 
  given point $(\ell,0,a_0,s_0)\in\jos^o$ (initial condition). 
  Let $s_1>s_0$ be the next point where  $\chi(s_1)=0$, and let the vector function 
  $(\chi(s),a(s))$ be well-defined on the segment $[s_0,s_1]$. The map 
  $\mcp:(\ell,a_0,s_0)\mapsto(\ell,a(s_1),s_1)$ of the first return to $\jos^o$  will be called the 
{\it Poincar\'e map} of the isomonodromic foliation $\mcg$. 
\end{definition} 
\begin{remark}
The Poincar\'e map $\mcp$, wherever defined, preserves the rotation number function, and hence, the phase-lock area portrait. This follows from constance of the rotation number along 
leaves of the  foliation $\mcg$. 
\end{remark}

The next theorem is our second main result. 
\begin{theorem} \label{tcgerm} 1) The Poincar\'e map $\mcp$ is well-defined and analytic on a neighborhood of 
the hyperplane $\{ a=0\}$ in $\jos$, including this hyperplane. 

2) Given an $\ell\in\rr$ and an $s>0$, let $s^*>s$ be the next zero after $s$ of the solution of 
the $\ell$-th Bessel equation vanishing at $s$.  The Poincar\'e map $\mcp$ sends the point $(\ell,0,0,s)$ to the point $(\ell, 0, 0, s^*)$.  In particular, it sends each 
$\beta_{\ell,k}=(\ell, 0, 0, s_{\ell,k})\in\jos$  to $\beta_{\ell,k+1}$. Here 
$s_{\ell,k}$ is the $k$-th positive zero of the Bessel function $J_\ell(s)$. 

3) For $\ell\in\zz$ let $\mcc_{\ell,k}$ and $\beta_{\ell,k}$ be the constriction curves 
and their landing points from Theorem \ref{cangerms}. The Poincar\'e map $\mcp$ sends 
each constriction curve germ $(\mcc_{\ell,k},\beta_{\ell,k})$  to the next 
germ $(\mcc_{\ell,k+1},\beta_{\ell,k+1})$. 
\end{theorem}

\subsection{Analytic extension and singularities of solutions of (\ref{isomnews}) and the 
Painlev\'e 3 equation (\ref{p3})} 
Our third main result is the following theorem. 

 \begin{theorem} \label{thext} 1) For every given $\ell\in\cc$ each solution $(\chi(s),a(s))$ of system 
  (\ref{isomnews}) extends to a meromorphic vector function on the universal cover of 
  punctured complex line $\cc^*_s=\cc\setminus\{0\}$. 
  
  2) Each its pole $s_0\in\cc^*$ (i.e., a pole of some its component) is  
  
  - either a zero with derivative one of the corresponding solution 
  $w(s)=\frac{a(s)}{2s\chi(s)}$ of Painlev\'e 3 equation (\ref{p3}); in this case $a(s)$ is holomorphic at $s_0$, $a(s_0)=\pm s_0$, and $\chi(s)$ has a simple pole at $s_0$ with residue $\pm\frac12$;
  
  - or a  pole of $w(s)$ with residue -1; in this case $a(s)$  has simple pole with residue $\pm s_0$ at $s_0$, and $\chi(s)$ is holomorphic 
   at $s_0$ and $\chi(s_0)=\mp\frac12$.

  3) For every initial condition $(\chi^*,a^*,s^*)$, $s^*\neq0$, corresponding to a linear system 
  (\ref{mchoy}) with trivial monodromy the corresponding solution $(\chi(s),a(s))$ of system (\ref{isomnews})  is meromorphic either on $\cc^*$, or on its double cover 
  $\cc^*_t$  given by $t\mapsto s=t^2$. More precisely, it 
  remains unchanged up to sign after analytic 
  extension along a simple circuit around the origin.  The corresponding solution $w(s)$ of Painlev\'e 3 equation (\ref{p3}) is always 
  meromorphic on $\cc^*$. 
  
  4) Statement 3) hold for every initial condition 
  $(0,a^*,s^*)$, where $(a^*,s^*)$ lies in the constriction submanifold $Constr_\ell$. 
  
  5) For every  $(0,a^*,s^*)$ such that $(a^*,s^*)$ lies in a constriction curve $\mcc_{\ell,k}$ the corresponding solution $(\chi(s),a(s))$ of (\ref{isomnews}) is meromorphic  on $\cc^*$. 
  \end{theorem}

 \subsection{The Poincar\'e map $\mcp$ and poles of  Painlev\'e 3 solutions}
 
 \begin{lemma} \label{crossect} \cite[lemma 6.7, p.5475]{bibgl}. 
 The graph of a non-identically-zero solution $(\chi(s), a(s))$ of system (\ref{isomnews}) crosses 
 the hyperplane $\{\chi=0\}$ exactly at those points $s=s_0\neq0$ where the corresponding solution 
 $w(s)=\frac{a(s)}{2s\chi(s)}$ of Painlev\'e 3 equation (\ref{p3}) has a pole with residue one.
 \end{lemma}
 Our main results imply the following new result.
 \begin{proposition} The Poincar\'e map $\mcp:\jos^o\to\jos^o$ is well-defined at a point 
 $(\ell, 0, a_0, s_0)$, if and only if the corresponding solution $\frac{a(s)}{2s\chi(s)}$ of 
  Painlev\'e equation (\ref{p3}) defined by the solution $(\chi(s), a(s))$ of (\ref{isomnews}) with 
  $(\chi(s_0), a(s_0))=(0,a_0)$ is holomorphic on a finite interval $(s_0,s_1)$, 
  has pole with residue one at $s_1$ and has no zeros with unit derivative in $(s_0,s_1)$. 
  \end{proposition}
  \begin{proof} The proposition follows from Lemma \ref{crossect} and Theorem \ref{thext}. 
  \end{proof}
  
  \begin{corollary} All the iterates of the Poincar\'e map $\mcp$ are defined at a given 
  point $(\ell, 0, a_0, s_0)$, if and only if the corresponding solution $w(s)$ of Painlev\'e 3 
  equation (\ref{p3}) has an  infinite sequence of real positive poles with residue 1 going to $+\infty$ and 
  no zero $s^*>s_0$ with unit derivative.
  \end{corollary}
  
  \begin{problem} Describe those $(\ell, 0, a_0, s_0)\in\jos^o$ for which $w(s)$  
  has an  infinite sequence of real positive poles with residue 1 going to $+\infty$ and 
  no zero $s^*>s_0$ with unit derivative.
  \end{problem}
  
  \begin{remark} Existence of an open subset of initial conditions for Painlev\'e 3 equation 
  (\ref{p3}) for  which $w(s)$ has an  infinite sequence of real positive poles with residue 1 going to $+\infty$  follows from  \cite[theorem 3.1]{ak} (implying an analogous statement on zeros of solutions of Painlev\'e 5 equations) and Novokshenov's Backlund transformation argument 
  \cite{nov23}. Each transcendental meromorphic on $\cc^*$ solution of Painlev\'e 3 equation 
  has infinite number of {\it complex} poles with  residue one; under an additional genericity assumption, it has infinite number of {\it complex} poles with both residues 
  $\pm1$ \cite[theorem 31.1]{GLS}. Existence of one-dimensional family of tronqu\'ee solutions,  
  which have no poles in a sector containing the real positive semiaxis, was proved in \cite{lidati}. 
  \end{remark}

 \subsection{Plan of the proof of main results}
  
  The proof of main results is based on the following characterization of the constrictions.
  
  \begin{proposition} \label{propoinc} \cite[proposition 2.2]{4} 
Consider the period $2\pi$ flow map $h=h^{2\pi}$ of system (\ref{josvec}) 
acting on the transversal coordinate $\theta$-circle $\{\tau=0\}$. (It is also the Poincar\'e map 
of the latter cross-section.) A point $(B,A;\omega)$ 
is a constriction, if and only if $\omega,A\neq0$ and $h=Id$. 
\end{proposition} 

\begin{proposition} \label{panal} In the coordinates $(\ell, a, s)$ on the parameter space of 
family (\ref{josvec}) the Poincar\'e map $h=h_{\ell, a, s}:S^1_\theta\to S^1_\theta$ 
is analytic as function on $\rr^3_{\ell, a, s}\times S^1_{\theta}$. 
\end{proposition}
The proposition follows from the classical theorem on analytic dependence of 
solution of analytic differential equation on parameters and initial condition.

  First in Subsection 2.1 we prove a part of Theorem \ref{cangerms}. Namely, we show that accumulation points 
  in the $s$-axis (if any) of the manifold $Constr_\ell$ have ordinates that are zeros of 
  the Bessel function $J_\ell$. We show that the germ of $Constr_\ell$ at each intersection point 
  lies in a regular germ of analytic curve. Then we deduce that the union of the closure 
  $\overline{Constr_\ell}\subset\rr_a\times(\rr_+)_s$  and  its image under the symmetry with 
  respect to the $s$-axis is an analytic submanifold. 
  
  In Subsection 2.3 we prove  Statements 1) and 2) of Theorem \ref{tcgerm}. To do this, 
  we blow up the space $\rr^4_{\ell,\chi,a,s}\setminus\{ s=0\}$ along the plane $\{a=\chi=0\}$. The complexified blown up manifold  will be denoted by $\wh{\cc^4}$. Its exceptional divisor 
   is $\Sigma:=\cc^2_{\ell,s}\times\cp^2_{[\chi:a]}\setminus\{ s=0\}$. By $\Pi\subset\wh\cc^4$ we denote the strict transform of the hyperplane $\{\chi=0\}$, which corresponds to model of Josephson junction.
 The line field given by (\ref{isomnews}) induces a holomorphic line field on the complexified blown-up manifold $\wh{\cc^4}$ that is tangent to the exceptional divisor $\Sigma$. The hypersurface $\Pi$ appears to be  its global cross-section. The restriction to $\Sigma$ of the line field appears to be given by the Riccati equation equivalent to the projectivization of the $\ell$-th Bessel equation. This yields a 
 geometric interpretation of the well-known fact that Painlev\'e equation (\ref{p3}) has a one-dimensional family of the so-called Bessel type solutions $w(s)=\frac{\frac{du(s)}{ds}}{u(s)}$, where $u(s)$ is $s^\ell$ times 
  a solution of the $\ell$-th Bessel equation. See  \cite[Section 10 in Chapter 12]{conte} and  Theorem \ref{bepainl} in Subsection 2.2 below. 
  
  It will be deduced from the above statements that the 
 first return Poincar\'e map $\mcp:\Pi\to\Pi$ is well-defined on a neighborhood of the intersection 
 $\Pi\cap\Sigma$, and its restriction to $\Pi\cap\Sigma$ is the map $(\ell,s)\mapsto(\ell,s^*)$, where 
 $s^*>s$ is the next (after $s$) zero of a solution of the Bessel equation vanishing at $s$. 
 
   Next, in Subsection 2.4 we show that each zero $s_{\ell,k}>0$ of Bessel function $J_\ell$ corresponds to an accumulation 
  point $\beta_{\ell,k}=(0,s_{\ell,k})$ of some constriction curve: a component of the manifold $Constr_\ell$.
   To do this, we first show that at least some $\beta_{\ell,m}$ is an accumulation point. This is done by using the above-mentioned part of Theorem \ref{cangerms} 
   (already proved in Subsection 2.1)  and 
   results of Klimenko and Romaskevich \cite{RK} on Bessel asymptotics of boundaries of the phase-lock areas. Then we deduce that each $s_{\ell,k}$ 
   corresponds to a constriction curve. This is done by using Statements 1) and 2) of 
   Theorem \ref{tcgerm} (already proved in Subsection 2.3) and  applying iterates of the Poincar\'e map $\mcp$ 
   (and its inverse) sending the above accumulation point $\beta_{\ell,m}$ to any other 
   $\beta_{\ell,k}$. 
  This will  finish the proofs of Theorems \ref{cangerms} and \ref{tcgerm}. 
  
    In Subsection 2.5 we prove Proposition \ref{conjs}. 
    
    In Subsection 2.6 we prove Theorem \ref{thgarl} and Corollary \ref{cqueer}.
    
  In Subsection 2.8 we prove Theorem \ref{thext}.  The proof uses Stokes phenomena theory and isomonodromicity of family of linear systems (\ref{mchoy}) along solutions of system (\ref{isomnews}). The corresponding background material on Stokes phenomena and isomonodromic families is briefly recalled in Subsection 2.7.

   \section{Proof of main results}
  \subsection{Intersection points of $\overline{Constr_\ell}$ with the $s$-axis: their regularity and landing components}
 Here we prove the following lemma.
 
 \begin{lemma} \label{lacreg} 
  1) For every $\ell\in\zz$ each finite accumulation point of the submanifold
 $Constr_\ell\subset\rr^2_+$ (if any) lies in the $s$-axis and has the type $\beta=(0,y)$, where $y>0$ is a  zero of the Bessel function $J_{\ell}$. 
 
 2) The subset $\widehat{Constr}_\ell\subset\rr^2$, which is the union  
  of the closure $\overline{Constr_\ell}\subset\rr^2$ and its symmetric 
  with respect to the $s$-axis,  
  is a one-dimensional analytic submanifold  intersecting the $s$-axis orthogonally at the above points $\beta$. 
  
  3) Each connected component of the submanifold $Constr_\ell$ landing at some point 
  $\beta$ of the $s$-axis accumulates to no  other point of the $s$-axis. 
  
  4) Each above landing component is projected onto the whole $a$-axis $\rr_+$.
  \end{lemma}
  The first step of the proof of Lemma \ref{lacreg} is the following proposition. 
 
 \begin{proposition} \label{necess} Let $\ell\in\zz$. Let a sequence of points in $Constr_\ell$ accumulate to 
a point $(0,y)$ of the $s$-axis. Then $J_\ell(y)=0$. 
\end{proposition}
\begin{proof} The submanifold $Constr_\ell$ is the intersection of the positive quadrant $\rr_+^2$ 
with the analytic subset 
\begin{equation}\mcm_\ell:=\{ h_{\ell, a, s}=Id\}\subset\rr^2,\label{monell}\end{equation}
see Propositions \ref{propoinc} and \ref{panal}. The germ at $(0,y)$ of the analytic subset $\mcm_\ell$ is one-dimensional, by assumption. 
Therefore, it is a finite union of germs of analytically parametrized curves. Let us calculate 
the derivative of the time $\tau$ flow maps $h^\tau:S^1\to S^1$ (and in particular, of the map 
$h=h^{2\pi}$) in the parameters $(a, s)$ from the equation in variations corresponding 
to (\ref{jostor}).  Set $\theta(\tau,\theta_0,a,s)=h^\tau(\theta_0)$ to be the solution of 
differential equation (\ref{jostor}) with the initial condition $\theta=\theta_0$ at $\tau=0$. One has 
\begin{equation}\theta|_{a=0}=\theta(\tau,\theta_0,0,s)=f(\tau):=\theta_0+\ell\tau+s\sin\tau, \ 
h|_{a=0}=Id,\label{theta0}\end{equation}
\begin{equation}\dot\theta'_a=\frac{d}{d\tau}\frac{\partial\theta}{\partial a}=
\cos\theta-a\sin\theta \theta'_a.\label{difeta}\end{equation}
Substituting $a=0$ and $\theta=f(\tau)$, see (\ref{theta0}), to (\ref{difeta}) and 
integrating (\ref{difeta}) with the initial condition $\theta'_a|_{\tau=0}=0$ yields 
$$\theta'_a(\tau)=\int_0^\tau\cos f(\tau)d\tau,$$
\begin{equation}\frac{\partial h}{\partial a}(\theta_0)|_{a=0}=\int_0^{2\pi}\cos(\theta_0+\ell\tau+s\sin\tau)d\tau.\label{frach}\end{equation} 

{\bf Claim 1.} {\it If $(0,y)$ is an accumulation point of the set $Constr_\ell$, then} 
\begin{equation}\frac{\partial h_{\ell, a, y}}{\partial a}(\theta_0)|_{a=0}=0 \text{ for every } \theta_0\in S^1.
\label{h'=0}\end{equation}

\begin{proof} Suppose the contrary: there exists a sequence of points $(a_k, s_k)\in Constr_\ell$ converging to $(0,y)$, as $k\to\infty$, but the derivative (\ref{h'=0}) is non-zero at  
some $\theta_0$. One has 
$$h_{\ell, a_k, s_k}(\theta_0)=\theta_0+a_k\frac{\partial h_{\ell, a, s_k}}{\partial a}(\theta_0)|_{a=0}+
O(a_k^2),$$
since $h_{\ell, 0, s}\equiv Id$, see (\ref{theta0}). Therefore, the latter right-hand side 
cannot be equal to $\theta_0$ for all $k$, since $a_k>0$. But $h_{\ell, a, s}=Id$ for all 
$(a, s)\in Constr_\ell$, by Proposition \ref{propoinc}. The contradiction thus obtained proves 
the claim. 
\end{proof}

System of identities  (\ref{h'=0}) for all $\theta_0$ is equivalent to the system of equalities
\begin{equation} \int_0^{2\pi}\cos(\ell\tau+y\sin\tau)d\tau=0; \ \int_0^{2\pi}\sin(\ell\tau+y\sin\tau)d\tau=0,\label{cossin}\end{equation} 
by (\ref{frach}) and cosine addition formula. 
The second identity in (\ref{cossin}) holds automatically, since the function $\sin(x)$ is odd and $2\pi$-periodic. 
The first integral in (\ref{cossin}) is equal to 
\begin{equation} \int_0^{2\pi}\cos(\ell\tau+y\sin\tau)d\tau= 2\int_0^{\pi}\cos(\ell\tau+y\sin\tau)d\tau=
2\pi J_\ell(-y),\label{besseld}\end{equation} 
see (\ref{besself}), since the function $\cos(x)$ is even  and $2\pi$-periodic. 
One has $J_\ell(-x)=(-1)^\ell J(x)$, see \cite[section 2.1, formula (2)]{watson}. Thus, (\ref{h'=0}) is equivalent to the 
equality $J_\ell(y)=0$. This together with Claim 1 proves Proposition \ref{necess}.
\end{proof}

\begin{proposition} \label{pgreg} Let $y\in\rr_+$ be such that $\beta=(0, y)$ is an accumulation point of the submanifold $Constr_\ell$. Then the germ at $\beta$ of the subset $\widehat{Constr}_\ell$ is a germ of regular curve orthogonal to the $s$-axis. 
\end{proposition}
\begin{proof} The submanifold $Constr_\ell$ consists of those $(a,s)\in\rr_+^2$ 
for which the difference $\Delta(a, s;\theta):=h_{\ell, a, s}(\theta)-\theta$ vanishes identically 
in $\theta\in\rr$. The Taylor series in $(a, s)$ of  the latter difference 
 is divisible by $a$, since it vanishes for $a=0$. Therefore, the function 
 $$\xi(a, s;\theta)=a^{-1}\Delta(a, s;\theta)$$ 
 is analytic. One has $J_\ell(y)=0$ (Claim 1) and 
\begin{equation}\xi(0, y;\theta_0)=0 \text{ for every } \theta_0,\label{xi0}\end{equation}
 by  the equality $\frac{\partial}{\partial a}\Delta(0, y;\theta_0)=0$, see (\ref{h'=0}). 
 
\begin{proposition} \label{propneq} For every $y>0$ such that $J_\ell(y)=0$ 
one has  
\begin{equation}\frac{\partial\xi}{\partial s}(0, y;\theta_0)\neq0 \ \ \text{ for  every }  \ 
 \theta_0\notin\frac{\pi}2+\pi\zz.\label{xineq}\end{equation}
 \end{proposition}
 \begin{proof} Inequality (\ref{xineq})  is equivalent to the inequality 
 \begin{equation}\frac{\partial^2\Delta}{\partial a\partial s}(0, y;\theta_0)\neq0 \ \ \text{ for  } \ 
 \theta_0\notin\frac{\pi}2+\pi\zz.\label{deltneq}\end{equation}
 Let us prove (\ref{deltneq}). Equations in variations on the similar derivative of 
  a solution $\theta(\tau,\theta_0, a, s)$ is obtained by 
  differentiation of equation (\ref{difeta}) 
  in $s$ and subsequent substitution $a=0$ and $\theta'_s|_{a=0}=\sin\tau$ (which 
  is also found from the equation in variation on $\theta'_s$):
$$\frac{d}{d\tau}\frac{\partial^2\theta}{\partial a\partial s}=
-\sin\theta \theta'_s
=-\sin(\theta_0+\ell\tau+s\sin\tau)\sin\tau$$
$$=\frac12(\cos(\theta_0+(\ell+1)\tau+s\sin\tau)-
\cos(\theta_0+(\ell-1)\tau+s\sin\tau)).$$
Integrating the latter differential equation with zero initial condition at $\tau=0$ and substituting 
$\tau=2\pi$ yields
$$\frac{\partial^2\theta}{\partial a\partial s}(2\pi,\theta_0, 0, y)=\frac{\partial^2\Delta}{\partial a\partial s}$$
\begin{equation}=\frac12\int_0^{2\pi}(\cos(\theta_0+(\ell+1)\tau+y\sin\tau)-
\cos(\theta_0+(\ell-1)\tau+y\sin\tau))d\tau$$
$$=\pi\cos\theta_0(J_{\ell+1}(-y)-J_{\ell-1}(-y)),\label{besseld}
\end{equation}
by (\ref{besself}), the addition formula for cosine and the second equality in (\ref{cossin}) 
(which holds for every integer  number $\ell$). 
The Bessel function difference in the 
latter right-hand side is equal to $-2J'_\ell(-y)$, by well-known  formula, see \cite[p. 17, formula (2)]{watson}. The latter derivative is non-zero, since $J_\ell$ is a solution of the second order differential equation (Bessel equation) and $J_\ell(-y)=(-1)^\ell J_\ell(y)=0$. Finally, the right-hand side in 
(\ref{besseld}) is non-zero. This together with (\ref{besseld}) implies (\ref{deltneq}). Proposition 
\ref{propneq} is proved.
\end{proof}

The analytic subset $\mcm_\ell\subset\rr^2$ 
is symmetric with respect to the $s$-axis. This follows from the fact that changing sing of 
$\omega$ is equivalent to time reversing in dynamical system (\ref{josvec}), 
which does not change triviality of the Poincar\'e map (the time $2\pi$ flow map  
of the zero fiber $\{0\}\times S^1_\theta$). 
The subset $\mcm_\ell\setminus\{ a=0\}\subset\rr^2_{a,s}$ coincides with the union of 
the subset $Constr_\ell$ and its image under the symmetry with respect to the $s$-axis. 
Both latter subsets are contained in the analytic subset 
$\mathcal A_{\theta_0}:=\{ \xi(a,s;\theta_0)=0\}$ for every fixed 
$\theta_0\in\rr$. For $\theta_0\notin\frac\pi2+\pi\zz$, the germ of the subset 
$\mathcal A_{\theta_0}$ at an accumulation point $\beta=(0, y)$  of the set $Constr_\ell$ is a germ $(\gamma, \beta)$ of regular 
analytic curve transversal to the $s$-axis, since $J_{\ell}(y)=0$ (Proposition \ref{necess}) and by 
Proposition \ref{propneq}. It is orthogonal to the $s$-axis, by symmetry. 
The punctured germ $\gamma\setminus\{\beta\}$ 
 coincides with the germ at $\beta$ of the set  $\mcm_\ell\setminus\{ a=0\}$, since the latter is a one-dimensional submanifold in $\rr^2_{a,s}\subset\rr^2\setminus\{ a=0\}$ 
 whose germ at $\beta$ is contained in $\gamma$.  This implies the statements of 
 Proposition \ref{pgreg}.
\end{proof}

\begin{proof} {\bf of Lemma \ref{lacreg}.} The submanifold $Constr_\ell$ accumulates to 
no point of the $a$-axis. For non-accumulation to the origin  see \cite[lemma 4.14]{bibgl}. 
Non-accumulation to points $(a,0)$ with $a>0$ is proved in  \cite[pp. 5459--5460, Case (1)]{bibgl}. (Formally speaking, the argument in loc. cit. proves non-accumulation to $(a,0)$ 
of sequence of points lying in the same component in $Constr_\ell$. But in fact, the argument is valid for every sequence in $Constr_\ell$.)   The germs of the set 
$\mcm_\ell\setminus\{ a=0\}$  at  points of its accumulation to 
 the $s$-axis  coincide with those of  the   analytic subset $\mathcal A_{\theta_0}$, 
 $\theta_0\notin\frac{\pi}2+\pi\zz$, punctured at the latter accumulation points: see the above proof of  Proposition \ref{pgreg}. They 
are regular at the base (accumulation) points, as is $\mathcal A_{\theta_0}$, which 
was proved above. 
The intersection of the set 
$\mcm_\ell\setminus\{ a=0\}$ with the positive quadrant coincides with $Constr_\ell$. 
This implies that the union of the closure 
$\overline{Constr_\ell}$ and its symmetric image with respect to the $s$-axis is an 
analytic submanifold in $\rr^2$.  Its orthogonality to the $s$-axis at their intersection points is 
obvious. The coordinate $a$ is unbounded from above on each connected component of the submanifold $Constr_\ell\subset\rr_+^2$, by 
\cite[theorem 1.12]{bibgl}. This implies that each component landing at some point in the $s$-axis 
is projected to the whole $a$-axis $\rr_+$ and cannot  start and end at points of the $s$-axis. This implies  Statements 3) and 4) of Lemma \ref{lacreg} and finishes the proof of the lemma.
\end{proof}

\subsection{Bessel type solutions of Painlev\'e 3 equation} 

\begin{theorem} \label{bepainl} \cite[Section 10, Chapter 12]{conte} For every $\ell\in\rr$ 
the corresponding  Painlev\'e 3 equation   (\ref{p3}) 
  admits a one-parameter family of solutions of type 
 $$w(s)=\frac{\frac{du(s)}{ds}}{u(s)},  \ \  \ u(s)=s^\ell(C_1J_{\ell}(s)+C_2Y_\ell(s)).$$
 Here $J_\ell(s)$ and $Y_\ell(s)$ are  linearly independent solutions of Bessel equation 
\begin{equation}\frac{d^2v}{ds^2}+\frac1s\frac{dv}{ds}+(1-\frac{\ell^2}{s^2})v=0.\label{besseq}
\end{equation}
 Namely, $J_\ell$ is the $\ell$-th Bessel function 
 $$J_\ell(x)=\sum_{p=0}^{\infty}\frac{(-1)^p}{\Gamma(p+1)\Gamma(p+\ell+1)}
 \left(\frac x2\right)^{2p+\ell}.$$
 (It is well-known that for $\ell\in\zz$, $J_\ell$ is also given by formula (\ref{besself}).) 
\end{theorem}

\begin{proposition} \label{ppric} A function $v(s)$ satisfies Bessel equation (\ref{besseq}), 
if and only if the function $u(s)=s^\ell v(s)$ satisfies the equation
\begin{equation}u''+\frac{1-2\ell}{s}u'+u=0\label{bsu}\end{equation}
In this case the function 
\begin{equation} w(s)=\frac{\frac{du(s)}{ds}}{u(s)}\label{wsdef}\end{equation}
satisfies the Riccati equation 
\begin{equation}w'=\frac{2\ell-1}sw-w^2-1.\label{ricw}\end{equation}
Each solution of (\ref{ricw}) satisfies Painlev\'e 3 equation (\ref{p3}). 
\end{proposition}
\begin{proof} Writing $v=us^{-\ell}$ and substituting this to 
Bessel equation (\ref{besseq}) yields 
$$s^{-\ell}u''-2\ell s^{-\ell-1}u'+\ell(\ell+1)s^{-\ell-2}u+\frac1s(s^{-\ell}u'-\ell s^{-\ell-1}u)+
(1-\frac{\ell^2}{s^2})s^{-\ell}u=0,$$
which is equivalent to (\ref{bsu}). Differentiating (\ref{wsdef}) and substituting (\ref{bsu}) yields
$$w'(s)=-w^2+\frac{u''(s)}{u(s)}=-w^2+\frac{2\ell-1}sw(s)-1.$$ 
Thus, Riccati equation (\ref{ricw}) is the projectivization of a linear system equivalent to 
(\ref{bsu}). This implies that each its solution is of type (\ref{wsdef}) with $u(s)=s^\ell v(s)$, 
where $v$ satisfies Bessel equation (\ref{besseq}). Hence, each its solution is a 
solution of Painlev\'e 3 equation (\ref{p3}), by 
Theorem \ref{bepainl}. 
 \end{proof}

\subsection{Blow up and domain of the Poincar\'e map $\mcp$}

\begin{convention}
We will be dealing with the plane $\cc^2_{\chi,a}$ blown up at the origin. The blown up plane will be denoted by $\wh\cc^2$.  The exceptional divisor 
(the pasted projective line) is naturally identified with the projectivization of the space 
$\cc^2_{\chi,a}$: the projective line $\cp^1$ equipped 
with homogeneous coordinates $[\chi:a]$. Since formally speaking, the functions 
$\chi$ and $a$ on $\wh\cc^2$ vanish on the exceptional 
divisor  $\cp^1\subset\wh\cc^2$, we rename its canonical homogeneous 
coordinates $[\chi:a]$ by $[y_1:y_2]$, in order to avoid confusion. Similarly, the space 
$\cc^4_{\ell,\chi,a,s}$ blown up along the $(\ell,s)$-space $\{\chi=a=0\}$ will be denoted by 
$\wh\cc^4$. Its exceptional divisor will be denoted by  
$$\Sigma:=\cp^1_{[y_1:y_2]}\times\cc^2_{\ell,s}.$$
\end{convention}

\begin{proposition} \label{blowric} Let $\Pi\subset\wh\cc^4$ denote the strict transform of the 
hyperplane $\{ \chi=0\}$ (which corresponds to complexified model of Josephson  junction) 
under the above blow up. 

1) System (\ref{isomnews}) lifts to a well-defined line field $\mathcal V$ on 
$$\wh\cc^{4,0}:=\wh\cc^4\setminus\{ s=0\}:$$ 
a non-autonomous ordinary differential equation  on 
$\cc_\ell\times\wh\cc^2$-valued function in $s$ with constant $\ell$-component. The exceptional  divisor $\Sigma$ is tangent to the line field $\mathcal V$. Its intersection with $\Pi$ is the plane 
 $\cc^2_{\ell,s}=\{[0:1]\}\times\cc^2_{\ell,s}$.

2) The restriction  to $\Sigma$ of the line field $\mathcal V$ is given by the Riccati equation 
on the function $\Psi(s)=\frac{y_2(s)}{y_1(s)}$ that is the projectivization of linear system 
 \begin{equation}\left(\begin{matrix} \dot y_1\\ \dot y_2\end{matrix}\right)=\left(\begin{matrix} 
-\frac\ell s & \frac1{2s}\\ -2s & \frac\ell s\end{matrix}\right)\left(\begin{matrix} y_1 \\ y_2\end{matrix}
\right).\label{matrex}\end{equation}
Namely,  $\Psi(s)$ is a solution of the above-mentioned Riccati equation, if and only if 
$\Psi(s)=\frac{y_2(s)}{y_1(s)}$  where $y=(y_1(s),y_2(s))$ is a solution of 
system (\ref{matrex}).  

3) Along each solution $y(s)$ of  system (\ref{matrex}) the corresponding function 
\begin{equation}w(s):=\frac{y_2(s)}{2sy_1(s)}\label{wy12}\end{equation}
satisfies Riccati equation (\ref{ricw}) and   Painlev\'e 3 equation (\ref{p3}).

 4) The  line field $\mathcal V$ restricted to $\Sigma$ is 
 transversal to the hypersurface $\Pi$ at their intersection points. 
 \end{proposition}
 
 \begin{proof} System (\ref{isomnews}) considered as a differential equation 
 on vector function $(\chi(s),a(s))$
  with constant parameter $\ell$  has right-hand side that 
  vanished for $\chi=a=0$. This implies that system (\ref{isomnews}) lifts to a 
  holomorphic line field $\mathcal V$  on $\wh\cc^4$ that  
  is tangent to the exceptional  divisor 
  $\Sigma=\{\chi=a=0\}=\cp^1_{[y_1:y_2]}\times\cc^2_{\ell,s}$. Its restriction to $\Sigma$ is 
  given by  the projectivized linear part in $(\chi,a)$ of the right-hand side of (\ref{isomnews}): linear system 
  (\ref{matrex}), whose matrix in the right-hand side coincides with the Jacobian matrix in $(\chi,a)$ 
 at $(0,0,s)$ of the $(\chi,a)$-component of system (\ref{isomnews}).  The ratio $w(s)$, see (\ref{wy12}), satisfies Riccati equation (\ref{ricw}): 
    $$w'=-\frac1sw+\frac{(-2sy_1+\frac\ell sy_2)y_1-y_2(-\frac\ell sy_1+\frac1{2s}y_2)}{2sy_1^2}=
\frac{2\ell-1}sw-1-w^2.$$
Painlev\'e 3 equation (\ref{p3}) on  $w(s)$ can be proved in two possible ways:

a) It follows from  Proposition \ref{ppric}.
 
b) Take a solution $(\ell, \chi(s), a(s))$ of system (\ref{isomnews}) with the initial condition 
 $(\ell,\chi_0,a_0)$ at $s=s_0$. The corresponding solution $\frac{a(s)}{2s\chi(s)}$ of 
 Painlev\'e 3 equation (\ref{p3}) converges to $w(s)$,  as $(\chi_0,a_0)\to(0,0)$ so that 
 $\frac{a_0}{2s_0\chi_0}\to w_0:=w(s_0)$. Therefore, the limit $w(s)$ also satisfies Painlev\'e 3 equation (\ref{p3}).
 
  Statements 1) -- 3) of 
 Proposition \ref{blowric} are proved. Transversality statement 3) follows 
 immediately from the fact that linear system (\ref{matrex}) has matrix with off-diagonal 
 elements being non-zero  for every $s>0$. 
 \end{proof}

  Consider the standard vector field directing $\mcv$: 
 
  \begin{equation}\begin{cases}\ell'_s=0\\ \chi'=\frac{a-2\chi(\ell+2\chi a)}{2s}\\
  a'=-2s\chi+\frac as(\ell+2\chi a)\\
 s'=1\end{cases}.\label{isomn*}\end{equation}  
 
  Set 
  $$\Pi_+:=\text{ the real part of } \Pi\cap\{ s>0\},  \ \ X_+:=\Pi_+\cap\Sigma.$$ 
 
 \begin{proposition} \label{poincmap} 
  1) For every initial condition  $x_0=[0:1]\times(\ell, s_0)\in X_+$ the real forward orbit of field (\ref{isomn*}) 
 starting at $x_0$ meets $X_+$ at some  point $x_1=[0:1]\times(\ell, s_1)$, $s_1>s_0$ (take $x_1$ to be the first  point  of return to $X_+$). 
 
 2) The same holds for 
 every  $x_0\in \Pi_+$ lying in some neighborhood of $X_+$, and the 
 map $x_0\mapsto x_1$ is a well-defined germ of analytic Poincar\'e first return 
 map $\mcp:(\Pi_+,X_+)\mapsto (\Pi_+,X_+)$.

 3) In the case, when the above $x_0$, $x_1$ lie in $X_+$, one has $x_1=\mcp(x_0)$, if and 
 only if $0<s_0<s_1$ and $s_0$, $s_1$ are neighbor zeros of a solution of 
    Bessel equation (\ref{besseq}). 
 \end{proposition}
 \begin{proof} The flow of field (\ref{isomn*}) on $\Sigma$ is given by solutions $w(s)$ of 
 the Riccati equation (\ref{ricw}). Intersections of its orbit with $\Pi$ are poles of the solution 
 $w(s)$. Recall that  $w=\frac{u'}u$, see (\ref{wsdef}), where $u(s)=s^\ell v(s)$, $v(s)$ is 
 a solution of Bessel equation (\ref{besseq}). Therefore poles of $w$ are exactly 
 zeros of $v(s)$. The Bessel function $J_\ell(s)$ has infinite sequence of positive zeros. 
 For every other solution $v(s)$ of Bessel equation its zeros are intermittent with those of 
 $J_\ell$, by Sturm Separation Theorem. Therefore, each its solution $v$ has also infinite 
 sequence of positive zeros. The map sending one  zero of a solution 
 to the next one depends analytically on the parameter of solution, since no two zeros can collide: 
 a non-trivial solution of a second order differential equation cannot have multiple zeros. 
 This proves Statements 1) and 3). Statement 2) follows from transversality statement 4) of 
 Proposition \ref{blowric}.  \end{proof}
 
 \subsection{The Poincar\'e map on constriction curves. Proofs of Theorems \ref{cangerms} 
 and \ref{tcgerm}}
 
 \begin{proposition} \label{pkr} The submanifold $\widehat{Constr}_\ell\subset\rr^2$ 
 intersects the $s$-axis in at least one point (in fact, at an infinite number of points) 
 with positive ordinate.
  \end{proposition}
 
 \begin{proof} The proof is analogous to the discussion in \cite[subsection 4.3]{bibgl}. 
 It is based on the Bessel asymptotics of the boundary curves of phase-lock areas proved by A.Klimenko and O.Romaskevich \cite{RK}. Let us recall their result. 
The boundary of the phase-lock area $L_r$ consists of two curves $\partial L_{r,0}$, $\partial L_{\ell,\pi}$, corresponding to those parameter values, for which the Poincar\'e map of the corresponding dynamical system (\ref{josvec}) acting on the circle $\{\tau=0\}$ (i.e., the time 
$2\pi$ flow map) has fixed points $0$ and $\pi$ respectively. These  are graphs 
$$\partial L_{r,\alpha}=\{ B= G_{r,\alpha}(A)\}, \ G_{r,\alpha} \text{ are analytic functions on } \rr; \  \alpha=0,\pi.$$
 \begin{theorem} \cite[theorem 2]{RK}. There exist positive constants 
 $\xi_1$, $\xi_2$, $K_1$, $K_2$, $K_3$ such that the following statement 
 holds. Let $r\in\zz$, $A$, $\omega>0$ be such that 
 \begin{equation}|r\omega|+1\leq \xi_1\sqrt{A\omega}, \ \ A\geq \xi_2\omega.
 \label{preeq}\end{equation}
  Let $J_r$ denote the $r$-th Bessel function. Then 
  \begin{equation}\left|\frac1\omega G_{r,0}(A)-r+\frac1{\omega} 
  J_r\left(-\frac A{\omega}\right)\right|\leq\frac1A\left(K_1+\frac{K_2}{\omega^3}+
  K_3\ln\left(\frac A{\omega}\right)\right),\label{bbo}\end{equation}
  \begin{equation}\left|\frac1\omega G_{r,\pi}(A)-r-\frac1{\omega} 
  J_r\left(-\frac A{\omega}\right)\right|\leq\frac1A\left(K_1+\frac{K_2}{\omega^3}+
  K_3\ln\left(\frac A{\omega}\right)\right).\label{bbo2}\end{equation}
 \end{theorem}

 Let $u_1<u_2<\dots$ denote the sequence of points of local maxima of the modulus $|J_\ell(-u)|$: $u_k\to+\infty$, as $k\to\infty$.
 
 \begin{proposition} \label{mezhdu} For every $\ell\in\zz$, $a_0>0$  and every $k$ large enough dependently 
 on $\ell$, $a_0$, for every $a\in(0,a_0]$ there exists a $w_k=w_k(a)\in(u_k, u_{k+1})$ such that 
 $(a,w_k)\in Constr_\ell$. 
 \end{proposition}
 
 \begin{proof} In the coordinates $(a,s)$ inequalities (\ref{preeq}), 
 (\ref{bbo}) and (\ref{bbo2}) 
 can be rewritten for $r=\ell$ respectively as 
 \begin{equation} |\frac{\ell}a|+1\leq\frac{\xi_1}a\sqrt{s}, \ s\geq\xi_2,
 \label{preeqn}\end{equation}
 \begin{equation} \left|a G_{\ell,0}(\frac{s}a)-\ell+ a 
  J_\ell(-s)\right|\leq\frac{a}{s}\left(K_1+K_2a^3+
  K_3\ln(s)\right),\label{bbon*}\end{equation}
  \begin{equation} \left|a G_{\ell,\pi}(\frac{s}a)-\ell-a 
  J_\ell(-s)\right|\leq\frac{a}{s}\left(K_1+K_2a^3+
  K_3\ln(s)\right).\label{bbon2*}\end{equation}
  For every $k$ large enough the value $s=u_k$ satisfies inequality (\ref{preeqn}) 
  for all $a\in(0,a_0])$. Substituting $s=u_k$ to the right-hand side in (\ref{bbon*}) 
 transforms it to a sequence of functions of $a\in(0,a_0]$ 
 with uniform asymptotics $a(O(\frac1{u_k})+O(\frac{\ln{u_k}}{u_k}))$, as $k\to\infty$. 
 The values $|J_\ell(-u_k)|$ are known to behave asymptotically as $\frac1{\sqrt u_k}$ 
 up to a known constant factor, see asymptotic formula for $J_\ell(s)$, as $s\to+\infty$, in \cite[section 7.1, p. 195]{watson}. 
 Therefore, they dominate the  
 right-hand sides in (\ref{bbon*}) and (\ref{bbon2*}). This together with (\ref{bbon*}),  (\ref{bbon2*}) implies that 
 for every $k$ large enough, 
 set $A_k=\frac{u_k}{a}$, $\omega=a^{-1}$, 
  $$G_{\ell,0}(A_k)=\ell\omega-J_\ell(-u_k)(1+o(1)), \ G_{\ell,\pi}(A_k)=\ell\omega+J_\ell(-u_k)(1+o(1)),$$
  as $k\to\infty$, uniformly in $\omega>a_0^{-1}$, i.e., uniformly in $a\in(0,a_0]$. 
  This implies that the difference 
  $$\Delta(A):=G_{\ell,\pi}(A)-G_{\ell,0}(A) \text{ has the same sign, as } J_\ell(-u_k), \text{ at  } A=A_k,$$
   whenever $k$ is large enough, for every $a\in(0,a_0]$. 
  The extrema $u_k$ are intermittent maxima and minima of the function $J_\ell(-u)$, and the  
  values $J_{\ell}(-u_k)$ have intermittent signs. Therefore, the same holds for the 
  above differences $\Delta(A_k)$. Thus, for every $k$ large enough 
  the function $\Delta(A)$ has at least one zero $W_k$ lying between $A_k$ and 
  $A_{k+1}$. The latter zero corresponds to an intersection point of the boundaries $\partial L_{\ell,0}$, $\partial L_{\ell,\pi}$. The latter intersection point is 
 a constriction of the phase-lock area $L_\ell$. Hence, its abscissa is equal to $\ell\omega$, 
 by  \cite[theorem 1.4]{bibgl}. 
 Finally, we have shown that for every $k$ large enough for every $a\in(0,a_0]$, set 
  $w_k=w_k(a):=aW_k$,  the corresponding point $(a,w_k)$  lies in the constriction submanifold $Constr_\ell$. One has $w_k\in(u_k,u_{k+1})$, by construction. 
  Proposition \ref{mezhdu} 
  is proved.
  \end{proof}
  
  The family of points $(a,w_k(a))$ from Proposition \ref{mezhdu} 
  with $a\in(0,a_0]$ accumulates to a point $(0,w^*_k)$, $w^*_k\in[u_k,u_{k+1}]$. 
  Thus, the manifold  $Constr_\ell$ has infinitely many limit points $(0,w^*_k)$. 
  This proves Proposition \ref{pkr}.
  \end{proof}

\begin{proposition} \label{posled} Let $\beta_{\ell,k}=(0,s_{\ell,k})$ lie in $\widehat{Constr}_\ell\subset\rr^2_{a,s}$. 
Then $\beta_{\ell,k+1}=(0,s_{\ell,k+1})$ also lies in the manifold $\widehat{Constr}_\ell$, and the germ of 
$\widehat{Constr}_\ell$ at 
$\beta_{\ell,k+1}$ is the image of its germ at $\beta_{\ell,k}$ under the Poincar\'e map $\mcp$. 
If in addition, $k\geq2$, then $\beta_{\ell,k-1}$ also lies in $\widehat{Constr}_\ell$, and its germ at 
$\beta_{\ell,k-1}$ is the image of its germ at $\beta_{\ell,k}$ under the inverse $\mcp^{-1}$.  \end{proposition}

\begin{proof} The Poincar\'e map $\mcp$ sends $\beta_{\ell,k}$ to $\beta_{\ell,k+1}$, by Proposition 
\ref{poincmap}, Statement 3). It sends the germ of the manifold $\widehat{Constr}_\ell$ 
at $\beta_{\ell,k}$ to its germ at $\beta_{\ell,k+1}$. Indeed, it preserves $\ell$ and the conjugacy class of the 
underlying dynamical system on torus, and hence, property of  dynamical system to have identity Poincar\'e map. Therefore, it preserves the property to be a constriction, by Proposition \ref{propoinc}. This implies the 
first statement of the proposition, on the map $\mcp$. Its second statement, on $\mcp^{-1}$, 
is proved analogously.  
\end{proof}

\begin{proof} {\bf of Theorems \ref{cangerms}  and \ref{tcgerm}.} An intersection point from Proposition \ref{pkr} 
is some of the $\beta_{\ell,k}=(0,s_{\ell,k})$.  It is a regular point of the ambient component of 
the submanifold $\widehat{Constr}_\ell$, and it is  unique intersection point of the component 
with the $s$-axis. Their 
intersection is orthogonal. The component  projects to the whole $a$-axis. 
These statements follow from Lemma \ref{lacreg}. All the $\beta_{\ell,k}$ are intersection points. This follows from  Proposition \ref{posled}. Theorem \ref{cangerms} is proved. Theorem \ref{tcgerm} follows from Theorem \ref{cangerms} and 
Propositions \ref{poincmap}, \ref{posled}.   \end{proof}

  \subsection{Conjectures \ref{conj3d} and \ref{conjeach}. Proof of Proposition 
  \ref{conjs}}
  
  \begin{proof} {\bf of Proposition \ref{conjs}.} 
  Let $\mcc$ be a connected component of the submanifold $Constr_\ell\subset(\rr_+^2)_{a,s}$. 
  The coordinate $a=\omega^{-1}$ is unbounded from above on $\mcc$, see  
  \cite[theorem 1.12]{bibgl}. Therefore, its  projection to the $a$-semiaxis is a semi-infinite interval 
  $(a_0,+\infty)$. Let now $a$ be bounded on a semicurve $\mcc_-$ of the curve $\mcc$. The semicurve  $\mcc_-$  
   should tend to "infinity" in $\rr_+^2$,  since $\mcc$ is a submanifold in $\rr_+^2$: a closed subset in $\rr_+^2$ that is locally a submanifold. Therefore, either $\mcc_-$ has a limit point in the positive $s$-axis, or $s$ tends to infinity 
   along the semicurve $\mcc_-$ (Theorem \ref{cangerms}).   In the first case one has $\mcc=\mcc_{\ell,k}$ for some $k$, by Theorem \ref{cangerms}.  The second case is impossible, 
  by \cite[proposition 4.13]{bibgl}. Finally, $\mcc=\mcc_{\ell,k}$. This proves Proposition 
  \ref{conjs}. \end{proof}
  
  Suppose  now that the statement of  Conjecture \ref{conj3d} is true for a given $\ell\in\zz$: 
  each component $\mcc$ 
   of the constriction submanifold $Constr_\ell$ is bijectively projected onto the 
  $a$-semiaxis. Then $\mcc$  is continuously parametrized by $a\in\rr_+$, and hence, $a$  is bounded on at least one  side of the curve $\mcc$. 
  Therefore, $\mcc=\mcc_{\ell,k}$ for some $k$, by Proposition \ref{conjs}. 
  Hence, the statement of Conjecture \ref{conjeach} holds. Thus, Conjecture \ref{conj3d} implies 
  Conjecture \ref{conjeach}.
  
  \subsection{Constriction curves and components of three-dimensional phase-lock areas. 
  Proof of Theorem \ref{thgarl} and Corollary \ref{cqueer}}
  
    In the proof of Theorem \ref{thgarl} and Corollary \ref{cqueer} we use the following theorem and lemma. 
  
  \begin{theorem} \label{thpos} \cite[theorem 1.7]{bibgl} For every $\omega>0$ and $r\in\zz$ 
  each constriction $(B_0,A_0)$ in the corresponding two-dimensional phase-lock area $L_r(\omega)$ is {\bf positive,} that is, there exists a  neighborhood 
  $U=U(A_0)\subset\rr$ such that the punctured interval $B_0\times(U\setminus\{ A_0\})$ lies 
  in $Int(L_r(\omega))$. 
  \end{theorem}

  \begin{proof} {\bf of Theorem \ref{thgarl}.} Fix some $r\in\zz$ and $\omega>0$.  Consider the 
  line $\La:=\{ a=\omega^{-1}\}\subset\rr^2_{a,s}$. 
 We will identify its points with their $s$-coordinates. For every point 
 $x=(B_0,A_0;\omega)\in Int(L_r)$ consider those  
 constrictions in $\partial L_r(\omega)$ that lie  below $x$, i.e., with $A<A_0$. Let us denote 
 their subset by $Constr_{r, x}$.  Let $C_{max}\in Constr_{r,x}$ denote the upper of them; set $s_{max}:=s(C_{max})=\frac{A(C_{max})}{\omega}$. The  constrictions in $Constr_{r,x}$ 
  correspond to intersection points of the segment $[0,s_{max}]\subset\La$  with  constriction curves (treated as components of the constriction submanifold $Constr_r$). To each constriction 
   $\eta\in Constr_{r,x}$ we assign its {\it multiplicity,} which is equal to the index of the  intersection 
  $Constr_r\cap\La$  at $\eta$. For each constriction 
  curve $\mcc$ let $n(\mcc)\in\{0,1\}$ denote the sum of multiplicities of those 
  constrictions in $Constr_{r,x}$ (taken modulo two) that lie in $\mcc$. To each point $x\in Int(L_r)$ we put 
  into correspondence its {\it code:} 
  $$\text{code}(x):=\text{ the collection of those constriction curves } \mcc \text{ for which } 
  n(\mcc)=1,$$
  i.e., the collection of those constriction curves that contain odd number of constrictions 
 (with multiplicities) lying in $Constr_{r,x}$. 
  
 \begin{proposition} For every $r\in\zz$ the code is constant on each connected component  of 
 the three-dimensional phase-lock area $L_r$.
 \end{proposition}
 \begin{proof} As a point $x$ moves continuously inside $Int(L_r)$, the constrictions of the 
 two-dimensional phase-lock area $L_r(\omega(x))$ lying below $x$ either remain the same, 
 or some new constrictions are born, or some of them disappear\footnote{Conjecture \ref{conj3d} 
  implies that as $\omega$ varies continuously, birth (disappearance) of constrictions is impossible.}. A birth of constrictions comes from 
 a tangency point of the line $\{a=\omega^{-1}(x)\}\subset\rr^2_{a,s}$ with a constriction curve 
 $\mcc$. The number of new constrictions born from a tangency point is obviously even, and all of 
 them lie in $\mcc$. Therefore, birth of new constrictions does not change the code. 
 The case of disappearance of constrictions is treated analogously.
 \end{proof}
 
 \begin{proposition} 1)For every $r\in\zz$, $N\in\nn$ there exists a small  $\var>0$ such that  for every $a_0\in(0,\var)$ the line $\La=\{ a=a_0\}\subset\rr^2_{a,s}$ intersects each one of the curves $\mcc_{r,k}$, $k=1,\dots,N$, transversally at a  point $x_k=x_k(a_0)$ close to $\beta_{r,k}$ so that $x_1,\dots,x_N$ are the only points of intersection of the 
 segment $I:=\{a_0\}\times[0,x_N+\var]$ with the constriction manifold $Constr_r$. 
 
 2) For every $k=1,\dots,N$ let us identify $x_k$ with the corresponding constriction 
 of the two-dimensional phase-lock area $L_r(\omega)$, $\omega=a_0^{-1}$. 
On the component of  $Int(L_r(\omega))$ adjacent to $x_k$ 
 from above the code is identically equal to $(\mcc_1,\dots,\mcc_k)$.
 \end{proposition}
 
 The proposition follows immediately from Theorem \ref{cangerms}. It implies that for every 
 $r\in\zz$ infinitely many different codes are realized by points of the interior of the three-dimensional phase-lock area $L_r$. Therefore, $Int(L_r)$  
 consists of infinitely many components. Theorem \ref{thgarl} is proved.
 \end{proof}
 
 \begin{proof} {\bf of Corollary \ref{cqueer}.} Statement (i) of Corollary \ref{cqueer} follows 
 from Proposition \ref{conjs}. There exists a domain $V\subset U_{\mcc}$ adjacent to $\mcc$ such that the corresponding subset $\wt V_{\ell}\subset\{ B=\ell\omega\}$, see (\ref{uphase}), 
 lies in  a connected component of  $Int(L_\ell)$. Indeed, take an arbitrary $a_0>0$ such that the line $\La=\{ a=a_0\}$ intersects $\mcc$ transversally in at least one point $x$. There exists a semi-interval $J(x)\subset\La$ adjacent to $x$  
 such that $J(x)$ lies both in $U_\mcc$ and in the two-dimensional 
 phase-lock area $L_\ell(\omega(x))$, $\omega(x)=a_0^{-1}$
 (positivity of constrictions, see Theorem \ref{thpos}, and transversality). 
 We take the above interval $J$ to be the maximal one.  The domain $V$ we 
 are looking for is the union of  the intervals $J(x)$ for all the points $x$ of transversal intersection  $\La\cap\mcc$. 
The corresponding subset $\wt V_\ell\subset\{ B=\ell\omega\}$ lies in a connected component of the three-dimensional phase-lock area $Int(L_\ell)$.  Let us denote the latter component 
by $Comp(V)$. Let us show that it is adjacent to no curve $\mcc_k$. 

Making the intersection $\mcc\cap\La$ transversal (by choosing a generic $a_0$) 
and taking the above $x$ to be the lower 
point of the latter intersection, we get that $J_\ell(x)$ is a vertical line interval adjacent to $x$ 
from above. The code of every its point contains the curve $\mcc$, by construction. 
Thus, the code of each point in $\wt V_\ell$ contains $\mcc$. On the other hand, we 
have the following 
\begin{proposition} \label{peven} On each component in $Int(L_\ell)$
adjacent to a constriction curve $\mcc_{\ell,k}$, the code can contain no queer constriction 
curve $\mcc$. 
\end{proposition}
\begin{proof} Fix a component $W$ adjacent to $\mcc_{\ell,k}$   of the interior $Int(L_\ell)$. 
There exists a domain $S\subset\rr^2_{a,s}$ adjacent to $\mcc_{\ell,k}$ such that the corresponding subset $\wt S_{\ell}\subset\{ B=\omega\ell\}\subset\rr^3_{B,A;\omega}$, 
 see (\ref{uphase}), 
lies in $W$.  This is proved analogously to the above proof of Corollary \ref{cqueer}. For every $x\in S$ the corresponding point  
$(\ell a^{-1}(x), \frac{s(x)}{a(x)};a^{-1}(x))\in\wt S_\ell\subset W$ will  be identified with $x$ and also denoted by $x$. Fix a queer constriction curve $\mcc$ (if any). We have to show that the code of every $x\in W$ does not contain $\mcc$. It suffices to prove this statement just for some $x\in W$. 
Taking $x\in S$ to be close enough to the endpoint $\beta_{\ell,k}$ of the curve $\mcc_{\ell,k}$, 
we can achieve that the vertical segment in $\La$ connecting $x$ and its projection to the $a$-axis
 does not intersect $\mcc$. This follows from the assumption that  $\mcc$ is queer (hence, 
 accumulates to no point of the $s$-axis, by Theorem \ref{cangerms}). Then the code of the point 
 $x$ does not contain $\mcc$, which follows from definition. Proposition \ref{peven} is proved.
 \end{proof}
 
  The component in $Int(L_\ell)$ containing $\wt V_\ell$ is adjacent to no curve $\mcc_{\ell,k}$, 
 since its code contains $\mcc$ (Proposition \ref{peven}). This proves Corollary \ref{cqueer}.
 \end{proof}
 
\subsection{Background material. Linear systems with irregular singularities: Stokes phenomena and isomonodromic deformations}

\subsubsection{Irregular singularities and Stokes phenomena}
The following material on irregular singularities of linear systems and Stokes phenomena is contained in \cite{2, BUL, bjl, 12, jlp, sib}. 

Consider a two-dimensional linear system 
 \begin{equation}Y'=
\left(\frac K{z^2}+\frac Rz+O(1)\right)Y, \ \ \ 
Y=(Y_1,Y_2)\in\cc^2,\label{eqlin}\end{equation}
on a neighborhood of 0. Here  $K$ and $R$ are complex $2\times2$-matrices, $K$ has distinct eigenvalues $\la_1\neq\la_2$, and $O(1)$ is a holomorphic 
matrix-valued function on a neighborhood of $0$. 
Then we say that the singular point 0 of system (\ref{eqlin}) is {\it irregular non-resonant 
of Poincar\'e rank 1.}  The matrix $K$ is conjugate to  
$\wt K=\diag(\la_1,\la_2)$, $\wt K=\mathbf H^{-1}K\mathbf H$, $\mathbf H\in GL_2(\cc)$,  and one can achieve that 
$K=\wt K$ by applying the constant linear change (gauge transformation)
$Y=\mathbf H\mathbf{\wh Y}$. 

Recall that two systems of type (\ref{eqlin}) are {\it analytically equivalent} near the origin, 
if one can be transformed to the other by linear space coordinate  change 
$Y=H(z)\wt Y$), where $H(z)$ is a holomorphic $GL_2(\cc)$-valued function on a neighborhood 
of the origin. Two systems (\ref{eqlin}) are {\it formally equivalent,} if the above $H(z)$ exists in 
the class of invertible formal power series with matrix coefficients. 

System (\ref{eqlin}) is 
formally equivalent to a unique {\it formal normal form} 
\begin{equation}\wt Y'=\left(\frac{\wt K}{z^2}+\frac{\wt R}z\right)\wt Y, \ \wt K=\diag(\la_1,\la_2), 
\ \wt R=\diag(b_1,b_2),\label{nform}
\end{equation}
\begin{equation} \wt R \ \  \text{ is the diagonal part of the matrix } \ \  \mathbf H^{-1}R\mathbf H.
\label{resfo}\end{equation} 
The matrix coefficient $K$ in  system (\ref{eqlin}) and the corresponding  matrix $\wt K$ 
in (\ref{nform}) are called the 
{\it main term matrices}, and $R$, $\wt R$  the {\it residue matrices}. However the normalizing series $H(z)$ generically diverges. At the same time, there exists a covering of a punctured 
neighborhood of zero by two sectors $S_0$ and $S_1$ with vertex at 0 in which 
there exist holomorphic $GL_2(\cc)$-valued matrix functions $H_j(z)$, $j=0,1$, that are $C^{\infty}$ 
smooth on  $\overline S_j\cap D_r$ for some $r>0$, and such that the variable changes $Y=H_j(z)\wt Y$ 
transform (\ref{eqlin}) to (\ref{nform}). This Sectorial Normalization Theorem 
 holds for the so-called {\it Stokes sectors}. Namely, consider the rays issued from 0 and forming the set 
\begin{equation}\{\re\frac{\la_1-\la_2}z=0\}.\label{strays0}\end{equation}
 They are called {\it imaginary dividing rays} or {\it Stokes rays}.  
A sector $S_j$ is called a {\it Stokes sector,} if it contains  one imaginary dividing ray 
 and its closure does not contain the other one. 
 
 Let $W(z)=\diag(e^{-\frac{\lambda_1}z}z^{b_1},e^{-\frac{\lambda_2}z}z^{b_2})$ denote the canonical diagonal fundamental 
matrix solution of the formal normal form (\ref{nform}). The matrices 
$X^j(z):=H_j(z)W(z)$ are fundamental matrix solutions of the initial equation (\ref{eqlin}) 
defining solution bases in $S_j$ called the {\it canonical sectorial solution bases.} 
Here we choose the branches $W(z)=W^j(z)$ of the  matrix function $W(z)$ in $S_j$ so that $W^1(z)$ is obtained from $W^0(z)$ by counterclockwise analytic extension from $S_0$ to $S_1$. 
And  we define the branch $W^2(z)$ of $W(z)$ in $S_2:=S_0$  obtained from $W^1(z)$ by counterclockwise 
analytic extension from $S_1$ to $S_0$. This yields 
another canonical matrix solution $X^2(z):=H_0(z)W^2(z)$ of system (\ref{eqlin}) in $S_0$, which is obtained 
from $X^0(z)$ by multiplication from the right by the monodromy matrix $\exp(2\pi i\wt R)$ 
of the formal normal form (\ref{nform}).  Let $S_{j,j+1}$ denote the connected component 
of intersection $S_{j+1}\cap S_j$, $j=0,1$, that is crossed when one moves 
from $S_j$ to $S_{j+1}$ counterclockwise, see Fig. 5. The transition matrices $C_0$, $C_1$ 
between thus defined canonical solution bases $X^j$, 
\begin{equation} X^1(z)=X^0(z)C_0 \text{ on }  S_{0,1}; \ \ X^2(z)=X^1(z)C_1 \text{ on } S_{1,2},\label{stokes}\end{equation}
are called the {\it Stokes matrices.}  

\begin{remark} \label{rstokes}
A priori, the canonical fundamental matrix solution $W(z)$ 
of the formal normal form does not necessary have a single-valued branch, since the exponents 
$b_k$ may be non-integer. A  diagonal fundamental matrix solution $W(z)$ is uniquely defined 
up to multiplication by diagonal matrix. This renormalization of the matrix $W(z)$ in the above construction  transforms 
the initial Stokes matrix pair to a new Stokes matrix pair obtained from the initial one by conjugation by one and the same diagonal matrix. It is well-known that {\it two germs of linear systems of type (\ref{eqlin}) are analytically equivalent, if and only if they have the same formal normal form and 
their Stokes matrix pairs are conjugated by the same diagonal matrix 
(i.e.,  can be normalized  to be the same).} 
\end{remark}

\begin{example} Let  $A=\diag(\la_1,\la_2)$, and let  $\la_2-\la_1>0$. Then 
the imaginary dividing rays are the positive and negative imaginary semiaxes. 
The Stokes sectors $S_0$ and $S_1$ covering $\cc^*$ satisfy the following 
conditions:

- the sector $S_0$ contains the positive imaginary semiaxis, and its closure does 
not contain the negative one;

- the sector $S_1$ satisfies the opposite condition. See Fig. 5.

\noindent The Stokes matrices $C_0$ and $C_1$ are unipotent upper and lower triangular 
respectively. 
\end{example}

\begin{figure}[ht]
  \begin{center}
   \epsfig{file=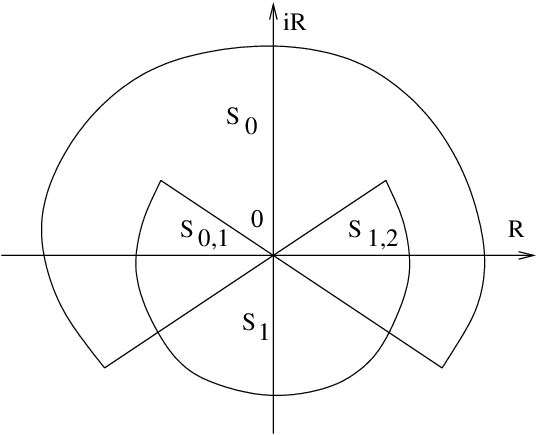}
    \caption{Stokes sectors in the case, when $\la_1-\la_2\in\rr$.}
  \end{center}
\end{figure} 

\begin{remark} Let $M_f$ denote the {\it formal monodromy} of system (\ref{eqlin}): 
 the monodromy matrix $\diag(e^{2\pi i b_1}, e^{2\pi i b_2})$ of its  formal normal form (\ref{nform}) with respect to a diagonal fundamental matrix solution $W(z)$.  Let $M$ denote 
 the monodromy matrix of  system (\ref{eqlin}) written in the basis given by the fundamental 
 matrix solution $X^0(z)$. It is well-known, see \cite[p.35]{12}, that 
\begin{equation}M=M_{norm}C_1^{-1}C_0^{-1}.\label{monst}\end{equation}
 \end{remark}
\subsubsection{Isomonodromic deformations}
We will deal with linear systems on $\oc$ of the type 
\begin{equation} Y'=\left(\frac K{z^2}+\frac Rz+ N\right)Y, \ \ Y\in\cc^2,\label{genlin}\end{equation}
where $K$, $R$, $N$ are complex $2\times2$-matrices such that each one 
of the matrices $K$, $N$ has distinct eigenvalues. Let $\la_{\infty1}$, $\la_{\infty2}$ denote the 
eigenvalues of the matrix $N$. System (\ref{genlin}) has two nonresonant irregular singular 
points of Poincar\'e rank 1: at the origin and at infinity. Namely, the variable change 
$\wt z=\frac1z$ transforms the infinity in $z$-coordinate to the origin in $\wt z$-coordinate, 
and system (\ref{genlin}) is transformed to a linear system with irregular singular point at the origin in the $\wt z$-coordinate. The corresponding Stokes rays, i.e., Stokes rays 
 "at infinity" of initial system (\ref{genlin}), are given by the equation 
 \begin{equation}\re((\la_{\infty2}-\la_{\infty1})z)=0\label{straysinf}\end{equation} 
 
 For every $p=0,\infty$ let $S_{p0}$, $S_{p1}$ be Stokes sectors of system (\ref{genlin}) at the 
 singular point $p$. Fix a point $z_0\in\cc^*=\cc\setminus\{0\}$ and paths $\alpha_{p}$ 
 going from the point $z_0$ to some point in $S_{p0}$. Let $X^{p0}(z)$ be the canonical 
 sectorial fundamental matrix solutions of system (\ref{genlin}) in sectors $S_{p0}$. Consider 
 their analytic extensions to $z_0$ along the paths $\alpha_{p}^{-1}$. Thus obtained germs of 
 fundamental matrix solutions $X^{p0}(z)$ at $z_0$ depend only on the homotopy classes of the paths 
 $\alpha_{p}$ in the class of paths with fixed starting point $z_0$ and free endpoint lying in 
 $S_{p0}$. Let $Q$ denote the transition matrix between the fundamental matrix solutions 
 $X^{p0}$, $p=0,\infty$, near $z_0$: 
 \begin{equation} X^{\infty0}(z)=X^{00}(z)Q.\label{transm}\end{equation}
 Let $C_{p0}$, $C_{p1}$ denote the Stokes matrices at $p$ defined by the sectorial 
 matrix solutions $X^{p0}(z)$. 
 
 \begin{remark} When we renormalize $X^{p0}(z)$ by multiplication from the right by diagonal matrix $D_p$ (which means similar renormalization of the corresponding branch of diagonal 
 fundamental matrix solution $W(z)=W^p(z)$ (defining $X^{p0}$) of formal normal form at $p$), 
 this transforms the initial Stokes matrix collections and transition matrix $Q$ to 
 $$\wt C_{pj}=D_p^{-1}C_{pj}D_p, \ \wt Q=D_0^{-1}QD_{\infty}.$$
 \end{remark}
 
 \begin{definition} \cite[definition 4.2]{FIKN}. 
 A continuous family of linear systems of type (\ref{genlin}) is {\it isomonodromic,} if one can choose continuous family of canonical fundamental matrix solutions 
 $X^{p0}(z)$ at $z_0$ defined by continuous families of Stokes sectors $S_{p0}$ and paths 
 $\alpha_p$ so that the monodromy matrix in the basis given by $X^{00}$, the Stokes matrices $C_{pj}$ and the transition matrix $Q$ remain the same. 
  \end{definition} 
  
  \begin{remark} \label{constform} Isomonodromicity is equivalent to the condition saying that  the 
 formal monodromies $M_{fp}$ at singular points $p$, the Stokes matrices and the transition matrix remain the same: see (\ref{monst}). In a continuous family of linear systems (\ref{genlin}) 
 constance of formal monodromies is equivalent to constance of residue matrices of formal normal forms. It is well-known that if an isomonodromic 
  family of linear systems depends analytically on a parameter lying in a simply connected domain 
  in a complex manifold, then the corresponding fundamental matrix solutions  from the above 
  definition can be normalized to depend analytically on the parameter. See \cite[theorem 13.1 and its proof]{Bol18} for the Fuchsian case; in the case of irregular singularities the proof is analogous. 
 \end{remark}

  \subsection{Extension of solutions. Proof of Theorem \ref{thext}} 
   
  The first step of the proof of Theorem \ref{thext} is the following proposition. 
  
  \begin{proposition} \label{asympt} Fix an $\ell\in\cc$ and a path $\alpha:[0,1]\to\cc^*$ from 
  a point $s^*$ to a point $s_0$; $s^*, s_0\in\cc^*$. 
  Let a solution $(\chi(s), a(s))$ of system (\ref{isomnews}) extend holomorphically to every 
  $\alpha(t)$, $t\in[0,1)$, 
  along the path $\alpha:[0,t]\to\cc^*$, but do not extend to $s_0=\alpha(1)$ along the path 
  $\alpha$.  Let 
  $w(s)=\frac{a(s)}{2s\chi(s)}$ be the corresponding solution of Painlev\'e 3 equation (\ref{p3}). 
   Then one of the two following statements holds:
   
   (i) either $w(s_0)=0$ and $w'(s_0)=1$; in this case $a(s)$ is holomorphic at $s_0$, 
   $a(s_0)=\pm s_0$, and $\chi(s)$ has a simple pole at $s_0$ with residue $\pm\frac12$;
   
   (ii)  or $w(s)$ has a simple pole at $s_0$ with residue $-1$; in this case $a(s)$ is meromorphic near $s_0$ and has a simple pole with residue $\pm s_0$ at $s_0$; $\chi(s)$ is holomorphic 
   near $s_0$ and $\chi(s_0)=\mp\frac12$. 
   \end{proposition}
   \begin{proof} The variable change $u=\chi^{-1}$ transforms system (\ref{isomnews}) to 
  \begin{equation}\begin{cases} u'_s=\frac1{2s}(-au^2+2\ell u+4a)\\
  a'_s=\frac2{us}(a^2-s^2)+\frac{a\ell}s\end{cases}\label{prfield}
 \end{equation}
 Therefore, differentiating the formula 
 \begin{equation} w(s)=\frac{a(s)}{2s\chi(s)}=\frac{a(s)u(s)}{2s}\label{waus}\end{equation}
  yields
 $$w'=\frac1{2s}(\frac1{2s}(-a^2u^2+2\ell a u+4a^2)+\frac2s(a^2-s^2)+\frac{\ell au}s)-
 \frac{au}{2s^2}$$
 \begin{equation}=\frac{2a^2}{s^2}-1+ \frac{(2\ell-1)w}s-w^2.\label{eqwu}\end{equation}
 
 \begin{proposition} \label{meroot} Let $\alpha$, $s_0$ and $(\chi(s),a(s))$ be the same, as in the conditions of Proposition \ref{asympt}. Then 
 both $\chi^2(s)$ and $a^2(s)$ are meromorphic on a neighborhood of the point $s_0$. 
 \end{proposition}
 \begin{proof} Meromorphicity of the function $a^2(s)$ near $s_0$ 
 follows from equation (\ref{eqwu}) and meromorphicity of the function $w(s)$ on the universal cover over $\cc^*_s$ (Painlev\'e property). This together 
 with (\ref{waus}) implies meromorphicity of the function $\chi^2(s)=\frac{a^2(s)}{4s^2w^2(s)}$.  
 \end{proof}
 
 At least one of the values $\chi(s_0)$, $a(s_0)$ should be infinite: otherwise, the solution 
 $(\chi(s), a(s))$ would be holomorphic at $s_0$, by Existence and Uniqueness Theorem for ODEs.

    Case 1): $w(s)$ is holomorphic at $s_0$. This together with (\ref{eqwu}) implies that 
    $a^2(s)$ is holomorphic at $s_0$. Thus, $a(s_0)$ is finite, hence
   $\chi(s_0)=\infty$, see the above 
    discussion.   Therefore, $u(s)$ and $w(s)=\frac{a(s)u(s)}{2s}$ vanish at $s=s_0$. 
    Every solution of Painlev\'e 3 equation (\ref{p3}) 
    has derivative $\pm1$ at each its zero. This is well-known and 
    follows directly from (\ref{p3}): the polar part 
    of its right-hand side at a zero of $w(s)$ comes from $\frac{(w')^2}w-\frac1w$, 
    and it  should cancel out. Now equalities $w(s_0)=0$, $w'(s_0)=\pm1$ together with 
    (\ref{eqwu}) imply that either $a(s_0)=\pm s_0$ if $w'(s_0)=1$, or $a(s_0)=0$ if $w'(s_0)=-1$. 
    In the second case the functions $a^2(s)$, $u^2(s)$ are both holomorphic (Proposition  \ref{meroot}) and both vanish at 
    $s_0$. One has $a(s)=c\sqrt{s-s_0}(1+o(1))$, $u(s)=-2sc^{-1}\sqrt{s-s_0}(1+o(1))$, 
    as $s\to s_0$, due to the latter statement and the 
    asymptotics $w(s)=\frac{a(s)u(s)}{2s}=-(s-s_0)(1+o(1))$. But the above asymptotics of 
    the function $u(s)$ contradicts the first equation in (\ref{prfield}): 
    the left-hand side should tend to infinity, as 
    $s\to s_0$, while the right-hand side doesn't. Thus, the second case  is 
    impossible. Hence, $w(s_0)=0$, $w'(s_0)=1$, $a(s_0)=\pm s_0\neq0$, and $a^2(s)$ is 
    holomorphic at $s_0$ (Proposition \ref{meroot}). Hence, $a(s)$ is holomorphic at $s_0$.
    This together with holomorphicity of $w(s)$ and 
    equality $w(s)=\frac{a(s)u(s)}{2s}=(s-s_0)(1+o(1))$ implies holomorphicity of 
    the function $u(s)$, the equalities $u(s_0)=0$, $u'(s_0)=\pm2$ 
    and Statement (i) of Proposition \ref{asympt}. 
    
   Case 2): $w(s)$ has a pole at $s_0$. Then this pole is simple with residue $\pm1$, see 
    \cite[p.158]{GLS}. In the case, when the sign is $+$, one has 
    \begin{equation} w(s)=\frac1{s-s_0}+\frac{2\ell-1}{2s_0}+O(s-s_0),\label{polas}\end{equation} 
    see  \cite[p.158]{GLS}. This together with equation (\ref{eqwu}) implies that the 
    Laurent series polar parts 
    at $s_0$  of the $w$-terms in its left- and right-hand sides are equal, and hence, $a^2(s)$ is  holomorphic at $s_0$. This together with (\ref{waus}) implies that $u(s)\to \infty$, hence, 
    $\chi(s)\to0$, as $s\to s_0$.  If $a(s_0)\neq0$, then $a(s)$ is holomorphic at $s_0$ and 
    $\chi=\frac{a}{2sw}$ has a simple zero at $s_0$. Let us now suppose that $a(s_0)=0$. Then 
    $u(s_0)=\infty$, by (\ref{waus}), and thus, $\chi(s_0)=0$. Finally, the point $(a(s_0),\chi(s_0))$ 
    is finite, -- a contradiction. Therefore, the residue of the function $w(s)$ at $s_0$ is equal to 
    $-1$. This together with (\ref{eqwu})  implies that the function $a^2(s)$ is meromorphic 
    near $s_0$ with order two pole at $s_0$ and asymptotics 
    $$a^2(s)=\frac{s_0^2}{(s-s_0)^2}(1+o(1)), \text{ as } s\to s_0.$$
    Hence, $a(s)$ is meromorphic near $s_0$ and has simple pole with residue 
    $\pm s_0$ at $s_0$. Together with (\ref{waus}), this implies that $\chi(s)$ is holomorphic 
    near $s_0$ and $\chi(s_0)=\mp\frac12$. Thus, Statement (ii) of Proposition \ref{asympt} 
    holds. Proposition \ref{asympt} is proved.
    \end{proof}
    
    \begin{proof} {\bf of Theorem \ref{thext}.} Statements 1) and 2) of Theorem \ref{thext} 
    follow from Proposition \ref{asympt}. Let us prove its Statement 3). A solution $(\chi(s), a(s))$ 
     of (\ref{isomnews}) yields an isomonodromic family of linear systems (\ref{mchoy}), see 
     \cite[theorem 6.6]{bibgl}. 
     Therefore, there exist analytic families $X^{00}(z,s)$, $X^{\infty0}(z,s)$ of sectorial 
    matrix solutions in Stokes sectors $S_{00}$, $S_{\infty0}$ of systems (\ref{mchoy}) with $\chi=\chi(s)$, $a=a(s)$ ($S_{p0}=S_{p0}(s)$ depend on $s$ for $p=0,\infty$) such that 
    
    - the Stokes matrices at the origin and at infinity defined by $X^{p0}(z,s)$, $p=0,\infty$,  remain constant (independent on $s$);
    
    - the  transition matrix $Q$,  \ $X^{\infty0}(z,s)=X^{00}(z,s)Q$, see (\ref{transm}), between analytic extensions to $z_0=1$ of the fundamental matrix solutions $X^{p0}$ along paths $\alpha_p^{-1}=\alpha_{ps}^{-1}\subset\cc^*$ 
    (depending continuously on $s$)  is also  constant. 
    
    As $s$ is changed, the Stokes sectors turn and the paths $\alpha_{ps}$ defining the above analytic extension 
    are deformed continuously (homotopied with  fixed starting point $z_0$ and free  
    endpoints  lying in variable sectors $S_{00}$, $S_{\infty0}$). As $s$ makes complete tour around the origin, both sectors $S_{00}$, $S_{\infty0}$ make complete tours, but in opposite directions: see (\ref{strays0}), (\ref{straysinf}). 
    The system corresponding to the analytic extension of the solution $(\chi(s), a(s))$ along 
    the circuit in question has the same formal normal forms and Stokes matrices at zero and 
    at infinity (for appropriately normalized canonical sectorial solutions in $S_{p0}$). But the homotopy classes of the paths defining the (unchanged) transition  matrix are changed. 

\begin{proposition} \label{trivmon} Let a system (\ref{mchoy}) have trivial monodromy. 
Then the corresponding transition matrix $Q$ given by (\ref{transm}) is path-independent: 
it depends only on the normalization of canonical sectorial fundamental matrix solutions 
in $S_{p0}$ by 
multiplication by diagonal matrices. 
\end{proposition}
\begin{proof} Each solution of system (\ref{mchoy}) in question  is holomorphic on $\cc^*$, by triviality of monodromy. This implies that the result of its 
analytic continuation along a path $\alpha$ from some point to $z_0$ is independent on its homotopy class in the class of paths in $\cc^*$ with fixed endpoints. Therefore analytic extension of 
each matrix solution $X^{p0}$ to $z_0$, and thus, the transition matrix $Q$, are path-independent. 
\end{proof}
     
 \begin{remark}  \label{trivst} As was shown in \cite[proposition 4.6]{bibgl}, triviality of monodromy of a system (\ref{mchoy}) implies triviality of its Stokes matrices and formal monodromies at both singular  points $0$ and $\infty$.
 \end{remark}
 
 Let the initial condition $(\chi^*,a^*, s^*)$ correspond to a system (\ref{mchoy}) with trivial 
 monodromy. Consider the corresponding solution $(\chi(s), a(s))$ 
 of system (\ref{isomnews}), which yields an isomonodromic deformation of linear system in question.  
 As $s$ makes a circuit around the origin, the initial condition $(\chi^*,a^*,s^*)$ is transformed to another triple $(\wt\chi, \wt a, s^*)$ corresponding to a system (\ref{mchoy}) with trivial 
 monodromy, the same Stokes sectors $S_{pj}$, $p=0,\infty$, $j=0,1$ (they are completely 
 defined by $s^*$, which remains the same)  
  and the same transition matrix $Q$, which is path-independent (Proposition \ref{trivmon}). 
  
  \begin{proposition} \label{gaugeq} Systems (\ref{mchoy}) corresponding to the above  
  $(\chi^*,a^*, s^*)$ and $(\wt\chi, \wt a, s^*)$ either coincide, or differ by sign of the two 
  first cordinates: $\wt\chi=-\chi^*$, $\wt a=-a^*$. Or equivalently, the corresponding 
  systems (\ref{mchoy}) either are the same, or are transformed one to the 
  other by the constant diagonal gauge transformation $Y\mapsto\diag(1,-1)Y$. 
  \end{proposition}
  \begin{proof}  The paths defining the transition matrix $Q$ for the new system can be chosen 
  the same, as for the initial system (Proposition \ref{trivmon}). Therefore, taking appropriate 
  normalization of the sectorial matrix solutions $X^{00}$, $X^{\infty0}$ for each system 
  one can achieve equality of the transition matrices. The Stokes matrices are trivial (Remark 
  \ref{trivst}). The formal normal forms of both systems at each singular point 
  are the same. Indeed, 
  they are completely defined by $s$ and the  residue matrix eigenvalue $\ell$, which are the same for both systems, since $\ell$ is independent on $s$ (see also Remark \ref{constform}). Therefore, the formal normal forms, the Stokes and transition matrices
   of both systems are the same for appropriate normalizations of the sectorial matrix solutions 
   $X^{p0}(z)$, $p=0,\infty$. This together with \cite[proposition 2.5, p.319]{JMU1} 
   implies that the systems are gauge equivalent: this means that there exists a matrix $H\in GL_2(\cc)$ such that the variable change 
  $Y\mapsto HY$ transforms one linear system to the other. Conjugation by the matrix $H$ should  preserve the opposite triangular forms of the  matrices $K$ and $N$. Therefore, $H$ is diagonal, since each one of the latter matrices has distinct eigenvalues. Moreover, conjugation by $H$ should preserve equality of opposite off-diagonal terms $\chi$ 
  of the matrices $K$ and $N$. This implies that the ratio of its eigenvalues is a square 
  root of unity, and $H=\diag(1,\pm1)$ up to scalar factor. Then one has 
  $(\wt\chi,\wt a)=(\pm\chi^*,\pm a^*)$. This proves Proposition \ref{gaugeq}.
  \end{proof}
  
  Proposition \ref{gaugeq} implies (together with Statement 1) of Theorem \ref{thext}) 
   that the solution $(\chi(s),a(s))$ of (\ref{isomnews}) with 
  initial conditions corresponding to a system (\ref{mchoy}) with trivial monodromy either  is meromorphic on $\cc^*$, or changes sign after analytic extension along a simple circuit 
  around the origin.  Therefore, the latter analytic extension does not change the ratio 
  $w=\frac a{2s\chi}$. Thus, $w(s)$ is meromorphic on $\cc^*$. Statement 3) of Theorem 
  \ref{thext} is proved. 
  
Statement 3) holds for every initial condition $(0, a^*, s^*)$ such that 
$(a^*,s^*)\in Constr_\ell$, since it corresponds to a linear system with trivial monodromy:  for every 
  constriction $(B,A;\omega)$ the corresponding linear system (\ref{mchoy}) with 
  $\chi=0$, $\ell=\frac B\omega$, $s=\frac A\omega$, $a=\omega^{-1}$ has trivial monodromy, see  
    \cite[proposition 4.1]{bibgl}. This proves Statement 4). 
  
  Let us now prove Statement 5). Let  $(0, a^*, s^*)$ be an initial condition with  
$(a^*,s^*)\in\mcc_{\ell,k}$ for some $k\in\nn$. The 
corresponding solution $(\chi(s), a(s))$ is either meromorphic single-valued on $\cc^*$, 
or meromorphic double-valued with sign-reversing monodromy. This means that it 
changes sign after analytic extension along a simple counterclockwise circuit in $\cc_s^*$ around the origin. 
This follows from Statement 4). It remains to show that the solution is single-valued. 
Indeed, suppose the contrary: it has sign-reversing monodromy for certain $(a^*,s^*)\in\mcc_{\ell,k}$. Then it has sign-reversing monodromy for every $(a^*,b^*)\in\mcc_{\ell,k}$, by continuity and connectivity of the curve 
$\mcc_{\ell,k}$. In particular, this holds for $(a^*,s^*)\in\mcc_{\ell,k}$ arbitrarily close to 
$\beta_k=(0,s_{\ell,k})$. Passing to the rescaled 
coordinates $(\wt y_1,\wt y_2)$, $(\chi,a):=(a^*\wt y_1,a^*\wt y_2)$, and taking the limit of 
thus rescaled system (\ref{isomnews}) and the solution in question with $\chi(s^*)=0$, $a(s^*)=a^*$, as 
$(a^*,s^*)\to(0,s_{\ell,k})$,  we get linear system (\ref{matrex}) and its solution $y=(y_1,y_2)(s)$ 
with $y(s_{\ell,k})=(0,1)$. The limit solution $y(s)$ should also have sign-reversing monodromy, by construction. 
Hence, -1 is an eigenvalue of the monodromy operator of  system (\ref{matrex}). 
But its monodromy operator is unipotent: its eigenvalues are equal to exponent of 
($2\pi i$ times the eigenvalues of its residue matrix at the origin); the residue eigenvalues  are $\pm\ell\in\zz$. The contradiction thus obtained finishes 
  the proof of Theorem \ref{thext}.
  \end{proof}
  
\section{Acknowledgements}
I am grateful to  V.M.Buchstaber and Yu.S.Ilyashenko for attracting 
my attention to problems on model of Josephson effect and helpful discussions. 
I am grateful to Yu.P.Bibilo, V.I.Gromak, V.Yu.Novokshenov for helpful discussions and to Yu.P.Bibilo for reading the paper and helpful remarks.


\begin{thebibliography}{}

\bibitem{ar} Anderson, P. W.; Rowell, J. M. {\it Probable observation of the Josephson tunnel effect.} Phys. Rev. Lett. \textbf{10 (6)} (1963): 230--232. 

\bibitem{ak} Andreev, F.; Kitaev, A. {\it Connection formulae for asymptotics of the fifth Painlev\'e  transcendent on the real axis.} Nonlinearity, \textbf{13} (2000), 1801--1840.

\bibitem{arn} Arnold, V. I. {\it Geometrical Methods in the Theory of Ordinary Differential Equations.} Second edition. Grundlehren der Mathematischen Wissenschaften 
[Fundamental Principles of Mathematical 
Sciences], 250. Springer-Verlag, New York, 1988.

\bibitem{2} Arnold, V. I.; Ilyashenko, Yu. S. {\it Ordinary differential equations.}  In: Dynamical Systems I, Encyclopaedia Math. Sci. (1988), 1--148.


\bibitem{BUL}
    W.~Balser, W.~B.~Jurkat, D.~A.~Lutz, {\it A general theory of invariants for meromorphic differential equations. I. Formal invariants}. 
		Funkcialaj Ekvacioj, {\bfseries 22}:2 (1979), 197--221.
		

\bibitem{bjl} Balser, W.; Jurkat, W.B.; Lutz, D.A.  {\it Birkhoff  invariants
and  Stokes'  multipliers  for meromorphic linear  differential
equations.}  J.	Math. Anal. Appl. \textbf{71} (1979), No. 1, 48--94.

\bibitem{bar} Barone, A.; Paterno, G. {\it Physics and Applications of the Josephson Effect.} John Wiley and
Sons, New York--Chichester--Brisbane--Toronto--Singapore, 1982.

\bibitem{bibgl} Bibilo, Yu.; Glutsyuk, A. {\it On families of constrictions in model of overdamped Josephson junction and Painlev\'e 3 equation.}  - Nonlinearity, \textbf{35} (2022), 5427--5480.


\bibitem{bibgl2} Bibilo, Y.; Glutsyuk, A.A. {\it On family of constrictions in model of overdamped Josephson 
junction.} 
Russ. Math. Surveys. \textbf{76:2} (2021), 360--362.


\bibitem{bbb} Bizyaev, I.A.; Borisov, A.V.; Mamaev, I.S. {\it The Hess--Appelrot case and
quantization of the rotation number.} Reg. Chaot. Dyn., \textbf{22:2} (2017),
180--196.

\bibitem{Bol89} Bolibruch, A.A. {\it The Riemann--Hilbert problem on 
the  complex projective  line.} [In Russian.] 
 Mat.  Zametki,  \textbf{46} (1989), No. 3, 118--120.

\bibitem{Bol7} Bolibruch, A. {\it Inverse problems for linear differential equations with meromorphic coefficients}. Isomonodromic Deformations and 
Applications in Physics (Montr\'eal, 2000), 
		CRM Proceeding and Lecture Notes, {\bfseries 31} (2002), 3--25.
	
\bibitem{Bol18} Bolibruch, A.A. {\it Inverse monodromy problems in analytic 
theory of differential equations.} [In Russian.]  Moscow, MCCME, 2018. 


\bibitem{bg} Buchstaber, V.M.; Glutsyuk, A.A. {\it On determinants of modified Bessel functions and entire solutions of double confluent 
Heun equations.} Nonlinearity, \textbf{29} (2016),  3857--3870.

\bibitem{bg2} Buchstaber, V.M.; Glutsyuk, A.A. {\it On monodromy eigenfunctions of Heun equations
and boundaries of phase-lock areas in a model
of overdamped Josephson effect.} Proc. Steklov Inst. Math., \textbf{297} (2017), 50--89. 


\bibitem{bktje} Buchstaber, V.M.;  Karpov, O.V.; Tertychniy, S.I. {\it Electrodynamic properties of a Josephson junction biased with a sequence of $\delta$-function pulses.} J. Exper. Theoret.
Phys., \textbf{93} (2001),  No. 6, 1280--1287.

\bibitem{bkt1} Buchstaber, V.M.;  Karpov, O.V.; Tertychnyi, S.I. {\it On properties of the differential
equation describing the dynamics of an overdamped Josephson junction,}  Russ. Math. Surveys, \textbf{59:2} (2004), 377--378.

\bibitem{buch2006} Buchstaber, V.M.; Karpov, O.V.;  Tertychnyi, S.I. {\it Peculiarities of dynamics of a Josephson
junction shifted by a sinusoidal SHF current.} [In Russian.]  Radiotekhnika i Elektronika, \textbf{51:6} (2006), 757--762.

\bibitem{buch2} Buchstaber, V.M.; Karpov, O.V.;  Tertychnyi, S.I. {\it The rotation number quantization effect}. Theoret and Math. Phys., \textbf{162} 
 (2010), No. 2,  211--221.

\bibitem{buch1} Buchstaber, V.M.; Karpov, O.V.; Tertychnyi, S.I. {\it The system on torus modeling the dynamics of Josephson junction.} 
Russ. Math. Surveys, \textbf{67}  (2012), No. 1, 178--180.

\bibitem{bt0} Buchstaber, V.M.; Tertychnyi, S.I. {\it Explicit solution family for the equation of the resistively shunted Josephson junction model.} Theoret. and Math. Phys., \textbf{176} (2013), No. 2, 965--986. 

\bibitem{bt1} Buchstaber, V.M.; Tertychnyi, S.I. {\it Holomorphic solutions of the double confluent Heun equation associated with the RSJ model of the Josephson junction.} Theoret. and Math. Phys., \textbf{182:3} (2015), 329--355.
 
 \bibitem{conte} Conte, R. (editor) {\it The Painlev\'e property: one century later.}  CRM Series in Mathematical 
 Physics. Springer,  1999.

\bibitem{FIKN}  Fokas, A.S.; Its, A.R.;  Kapaev, A.A.; Novokshenov, V.Yu. {\it Painlev\'e Transcendents: The Riemann-Hilbert Approach.} Amer. Math. Soc., 2006.


\bibitem{Foote} Foote, R.L., {\it Geometry of the Prytz Planimeter.} Reports on Math. Phys. \textbf{42:1/2} (1998), 249--271.

\bibitem{foott} Foote, R.L.; Levi, M.; Tabachnikov, S. {\it Tractrices, bicycle tire tracks, hatchet planimeters, and a 100-year-old conjecture.} 
Amer. Math. Monthly, \textbf{120} (2013), 199--216.


\bibitem{4} Glutsyuk, A.A.; Kleptsyn, V.A.; Filimonov, D.A.; Schurov, I.V. {\it On the adjacency quantization in an equation modeling the 
Josephson effect.} Funct. Analysis and Appl., \textbf{48} (2014), No. 4, 272--285.


\bibitem{g18} Glutsyuk A. {\it On constrictions of phase-lock areas in model
of overdamped Josephson effect and transition matrix
of the double-confluent Heun equation.} J. Dyn. Control Syst. \textbf{25}  (2019), 
 Issue 3,  323--349. 
 
 \bibitem{gn19} Glutsyuk, A.; Netay, I. {\it On spectral curves and complexified boundaries of phase-lock areas in a model of Josephson junction.}  -- J. Dyn. Control Systems, 26 
(2020), 785--820.

\bibitem{grauert} Grauert, H. {\it Ein Theorem der analytischen Garbentheorie und die Modulr\"aume komplexer Strukturen.} Inst. Hautes Etudes Sci., Publ. Math. 
\textbf{5} (1960), 5--64.

\bibitem{grh} Griffiths, Ph.; Harris, J., {\it Principles of algebraic geometry,} 
John Wiley $\&$ Sons, New York - Chichester - Brisbane - Toronto, 1978. 

\bibitem{GLS} Gromak, V.I.; Laine, I;  Shimomura, S. {\it Painlev\'e Differential Equations in the Complex Plane.} Walter de Gruyter, 
Berlin -- New York, 2002. 

\bibitem{LSh2009} Ilyashenko, Yu.S. {\it Lectures of the summer school ``Dynamical systems''.} Poprad, Slovak Republic, 2009.

\bibitem{ilguk} Ilyashenko, Yu.; Guckenheimer, J. {\it The duck and the devil: 
canards on the staircase.} Moscow Math. J., \textbf{1} (2001), No. 1, 27--47.

\bibitem{IRF} Ilyashenko, Yu.S.; Filimonov, D.A.; Ryzhov, D.A.  {\it Phase-lock effect for equations modeling resistively shunted
Josephson junctions and for their perturbations.}  Funct. Analysis and its Appl. \textbf{45} (2011), No. 3, 192--203.

\bibitem{12} Ilyashenko, Yu. S.;   Khovanskii, A. G. {\it Galois groups, Stokes operators, and a theorem of
Ramis.} Functional Anal. Appl., \textbf{24:4} (1990), 286--296.

\bibitem{J} Jimbo, M. {\it Monodromy Problem and the Boundary Condition for Some Painlev\'e Equations.} Publ. RIMS, Kyoto Univ. \textbf{18} (1982), Issue 3, 1137--1161.

\bibitem{JMU1} Jimbo, M.; Miwa, T.; Ueno, K. {\it Monodromy Preserving Deformations of Linear Ordinary Differential Equations with
Rational Coefficients (I).} Physica  D, {\bf 2} (1981), 306--352.

\bibitem{JMU2}  Jimbo, M.; Miwa, T.; Ueno, K. {\it Monodromy Preserving Deformation of Linear Ordinary Differential Equations
with Rational Coefficients II.} Physica D, {\bf 2} (1981), 407--448.

\bibitem{josephson} Josephson, B.D., {\it Possible new effects in superconductive tunnelling.} Phys. Lett., \textbf{1} (1962), No. 7, 
251--253. 

\bibitem{jlp} Jurkat, W.B.; Lutz, D.A.;   Peyerimhoff, A.   {\it Birkhoff
invariants and	effective  calculations	 for meromorphic linear
differential equations.} J. Math. Anal. Appl. \textbf{53} (1976), No. 2, 438--470.

\bibitem{krs} Kleptsyn, V.A.; Romaskevich, O.L.; Schurov, I.V. 
{\it Josephson effect and slow-fast systems.} [In Russian.]  Nanostuctures. Mathematical physics and Modelling, \textbf{8} (2013), 31--46. 

\bibitem{RK} Klimenko, A.V; Romaskevich, O.L. {\it Asymptotic properties of Arnold tongues
and Josephson effect.}  Mosc. Math. J., \textbf{14:2} (2014), 367--384.

\bibitem{lev} Levinson, Y. {\it Quantum noise in a current-biased Josephson junction.} 
Phys. Rev. B \textbf{67} (2003), 184504.

\bibitem{lich} Likharev, K.K. {\it Dynamics of Josephson junctions and circuits.} 
Gordon and Breach Science  Publishers, 1986.

\bibitem{lich-rus} Likharev, K.K. {\it Introduction to the dynamics of Josephson junctions.} 
[In Russian.]  Moscow, Nauka, 1985.

\bibitem{likh-ulr} Likharev, K.K.; Ulrikh, B.T. {\it Systems with Josephson junctions: Basic Theory.} [In Russian.] Izdat.
MGU, Moscow, 1978.

\bibitem{lidati} Lin Y., Dai D., Tibboel P., {\it Existence and uniqueness of tronqu\'ee solutions of the third and fourth 
Painlev\'e equations,} Nonlinearity \textbf{27} (2014), 171--186.

\bibitem{mcc} McCumber, D.E. {\it Effect of ac impedance on dc voltage-current characteristics of superconductor weak-link junctions,}  J. Appl. Phys., \textbf{39} (1968), No. 7,   3113--3118. 


\bibitem{nov23} Novokshenov, V.Yu. 
{\it Distribution of real poles of special Painlev\'e 3 equation.} To appear.


\bibitem{ok} Okamoto, K. {\it The Hamiltonians associated to the Painlev\'e equations.} In {\it The Painlev\'e Property}, R. Conte (ed.), Springer-Verlag New York Inc. 1999. Pages 735--787. 


\bibitem{schmidt} Schmidt, V.V., {\it Introduction to physics of superconductors.} [In Russian.] MCCME, Moscow, 2000. 

\bibitem{shap} Shapiro, S.; Janus, A.; Holly, S. {\it Effect of microwaves on Josephson currents in superconducting
tunneling,}  Rev. Mod. Phys., \textbf{36} (1964), 223--225.

\bibitem{sib} Sibuya,  Y.  {\it Stokes phenomena.}  Bull.
Amer. Math. Soc., \textbf{83} (1977),  1075--1077. 


\bibitem{stewart} Stewart, W.C., {\it Current-voltage characteristics of Josephson junctions.} Appl. Phys. Lett., \textbf{12} (1968), No. 8, 277--280. 


\bibitem{tert} Tertychnyi, S.I. {\it Long-term behavior of solutions of the equation $\dot\phi +\sin\phi = f$ with periodic $f$ and the modeling
of dynamics of overdamped Josephson junctions.} Preprint  https://arxiv.org/abs/math-ph/0512058. 

\bibitem{tert2} Tertychnyi, S.I. {\it The modeling of a Josephson junction and Heun polynomials.} 
Preprint https://arxiv.org/abs/math-ph/0601064.

\bibitem{watson} Watson, G.N., {\it A treatise on the theory of Bessel functions (2nd. ed.).}  Vol. 1,   Cambridge University Press, 1966. 
\end{thebibliography}
\end{document}